\documentclass[11pt,reqno]{amsart}%
\usepackage{graphicx,adjustbox}
\usepackage{stmaryrd}
\usepackage{mathtools, amsfonts, amssymb, amsthm}
\usepackage{enumerate, url}
\usepackage[margin=1in]{geometry}
\usepackage{caption}
\usepackage{algorithm}
\usepackage{algpseudocode}
\usepackage{mleftright}
\usepackage[normalem]{ulem}
\usepackage{hyperref}
\usepackage{tabu}
\usepackage{tikz-cd}

\usepackage{xcolor}
\hypersetup{
    colorlinks,
    linkcolor={red!50!black},
    citecolor={blue!50!black},
    urlcolor={blue!80!black}
}

\let\latexcirc=\circ
\newcommand{\ccirc}{\mathbin{\mathchoice
  {\xcirc\scriptstyle}
  {\xcirc\scriptstyle}
  {\xcirc\scriptscriptstyle}
  {\xcirc\scriptscriptstyle}
}}
\newcommand{\xcirc}[1]{\vcenter{\hbox{$#1\latexcirc$}}}
\let\circ\ccirc

\makeatletter
\def\Ddots{\mathinner{\mkern1mu\raise\p@
\vbox{\kern7\p@\hbox{.}}\mkern2mu
\raise4\p@\hbox{.}\mkern2mu\raise7\p@\hbox{.}\mkern1mu}}
\makeatother

\newtheorem{theorem}{Theorem}[section]
\newtheorem{proposition}[theorem]{Proposition}
\newtheorem{proposition/definition}[theorem]{Proposition/Definition}
\newtheorem{lemma}[theorem]{Lemma}
\newtheorem{corollary}[theorem]{Corollary}

\theoremstyle{definition}
\newtheorem{definition}[theorem]{Definition}

\newtheorem{example}[theorem]{Example}

\theoremstyle{remark}

\newcommand{\tp}{{\scriptscriptstyle\mathsf{T}}}
\newcommand{\F}{{\scriptscriptstyle\mathsf{F}}}

\newcommand{\lb}{\llbracket}
\newcommand{\rb}{\rrbracket}
\newcommand{\grad}{\nabla\!}
\newcommand{\Hess}{\nabla^2\!}
\newcommand{\eig}{\mathcal{E}}
\newcommand{\cay}{\mathcal{C}}
\newcommand{\qr}{\mathcal{Q}}

\let\O\undefined
\DeclareMathOperator{\Gr}{Gr}
\DeclareMathOperator{\V}{V}
\DeclareMathOperator{\GL}{GL}
\DeclareMathOperator{\O}{O}

\DeclareMathOperator{\tr}{tr}
\DeclareMathOperator{\proj}{proj}
\DeclareMathOperator{\im}{im}
\DeclareMathOperator{\rank}{rank}
\DeclareMathOperator{\diag}{diag}
\DeclareMathOperator{\spn}{span}

\DeclareMathOperator*{\argmin}{argmin}

\DeclareMathOperator{\sinc}{sinc}
\DeclareMathOperator{\vect}{vec}
\DeclareMathOperator{\T}{\mathbb{T}}
\DeclareMathOperator{\N}{\mathbb{N}}
\DeclareMathOperator{\B}{B}

\newcommand{\Rmnum}[1]{\expandafter\@slowromancap\Romannumeral #1@} 

\begin{document}
\title{Simpler Grassmannian optimization}
\author[Z.~Lai]{Zehua Lai}
\address{Computational and Applied Mathematics Initiative,
University of Chicago, Chicago, IL 60637-1514.}
\email{laizehua@uchicago.edu, lekheng@uchicago.edu}
\author[L.-H.~Lim]{Lek-Heng~Lim}
\author[K.~Ye]{Ke~Ye}
\address{KLMM, Academy of Mathematics and Systems Science, Chinese Academy of Sciences,
Beijing 100190, China}
\email{keyk@amss.ac.cn}

\begin{abstract}
There are two widely used models for the Grassmannian $\Gr(k,n)$, as the set of equivalence classes of orthogonal matrices $\O(n)/\bigl(\O(k) \times \O(n-k)\bigr)$, and as the set of trace-$k$ projection matrices $\{P \in \mathbb{R}^{n \times n} : P^\tp = P = P^2,\; \tr(P) = k\}$. The former, standard in manifold optimization, has the advantage of giving numerically stable algorithms but the disadvantage of having to work with equivalence classes of matrices. The latter, widely used in coding theory and probability, has the advantage of using actual matrices (as opposed to equivalence classes) but working with projection matrices is numerically unstable. We present an alternative that has both advantages and suffers from neither of the disadvantages; by representing $k$-dimensional subspaces as symmetric orthogonal matrices of trace $2k-n$, we obtain
\[
\Gr(k,n) \cong \{Q \in \O(n) : Q^\tp = Q, \; \tr(Q) = 2k -n\}.
\]
As with the other two models, we show that differential geometric objects and operations --- tangent vector, metric, normal vector, exponential map, geodesic, parallel transport, gradient, Hessian, etc --- have closed-form analytic expressions that are computable with standard numerical linear algebra. In the proposed model, these expressions are considerably simpler, a result of representing $\Gr(k,n)$ as a linear section of a compact matrix Lie group $\O(n)$, and can be computed with at most one \textsc{qr} decomposition and one exponential of a special skew-symmetric matrix that takes only $O\bigl(nk(n-k)\bigr)$ time. In particular, we completely avoid eigen- and singular value decompositions in our steepest descent, conjugate gradient, quasi-Newton, and Newton methods for the Grassmannian. 
\end{abstract}

\subjclass[2010]{14M15, 90C30, 90C53, 49Q12, 65F25, 62H12}

\keywords{Grassmannian, Grassmann manifold, manifold optimization}

\maketitle

\section{Introduction}\label{sec:intro}

As a manifold, the Grassmannian $\Gr(k,n)$ is just the set of $k$-planes in $n$-space with its usual differential structure; this is an abstract description that cannot be employed in algorithms and applications. In order to optimize functions $f :  \Gr(k,n) \to \mathbb{R}$ using currently available technology, one needs to put a coordinate system on $\Gr(k,n)$. The best known way, as discovered by Edelman, Arias, and Smith in their classic work \cite{EAS}, is to realize $\Gr(k,n)$ as a \emph{matrix manifold} \cite{AMS}, where every point on $\Gr(k,n)$ is represented by a matrix or an equivalence class of matrices and from which one may derive closed-form analytic expressions for other differential geometric objects (e.g., tangent, metric, geodesic) and differential geometric operations (e.g., exponential map, parallel transport) that in turn provide the necessary ingredients (e.g., Riemannian gradient and Hessian, conjugate direction, Newton step) for optimization algorithms. The biggest advantage afforded by the approach in \cite{EAS} is that a judiciously chosen system of extrinsic \emph{matrix coordinates} for points on $\Gr(k,n)$ allows all aforementioned objects, operations, and algorithms to be computed solely in terms of standard \emph{numerical linear algebra}, which provides a ready supply of stable and accurate algorithms \cite{GVL} with high-quality software implementations \cite{LAPACK}. In particular, one does not need to solve any differential  equations numerically when doing optimization on matrix manifolds \`a la \cite{EAS}.

\subsection{Existing models} There are two well-known models for $\Gr(k,n)$ supplying such matrix coordinates --- one uses orthogonal matrices and the other projection matrices. In optimization, the by-now standard model (see, for example, \cite{Drei,JD,MS,SS,WY}) is the one introduced in \cite{EAS}, namely,
\begin{equation}\label{eq:EAS}
\Gr(k,n) \cong \O(n)/\bigl(\O(k) \times \O(n-k)\bigr) \cong \V(k,n)/\O(k),
\end{equation}
where  $\V(k,n) \coloneqq \{V \in \mathbb{R}^{n \times k} : V^\tp V = I \} \cong \O(n)/\O(n-k)$ is the Stiefel manifold. In this homogeneous space model, which is also widely used in areas other than optimization \cite{BCN, BN, GS, HHSLS, MD, MKJS, ZT}, a point $\mathbb{V} \in \Gr(k,n)$, i.e., a $k$-dimensional subspace $\mathbb{V} \subseteq \mathbb{R}^n$, is represented by its orthonormal basis, written as columns of a matrix $V = [v_1,\dots,v_k] \in \V(k,n)$. Since any two orthonormal bases $V_1, V_2 \in \V(k,n)$ of $\mathbb{V}$ must be related by $V_1 = V_2 Q$ for some $Q \in \O(k)$, such a representation is not unique and so this model requires that we represent $\mathbb{V}$ not as a single $n \times k$ orthonormal matrix but as a whole equivalence class $\lb V \rb \coloneqq \{VQ \in \V(k,n) : Q \in \O(k) \}$ of orthonormal bases of $\mathbb{V}$. A brief word about our notations: Throughout this article, we adopt the convention that a vector space  $\mathbb{V} \in \Gr(k,n)$ will be typeset in blackboard bold, with the corresponding letter in normal typeface $V \in  \V(k,n)$ denoting an (ordered) orthonormal basis. Equivalence classes will be denoted in double brackets, so $\lb V \rb = \mathbb{V}$. Diffeomorphism of two smooth manifolds will be denoted by $\cong$.

It is straightforward to represent a point $\mathbb{V} \in \Gr(k,n)$ by an actual matrix as opposed to an equivalence class of matrices. Since any subspace $\mathbb{V}$ has a unique orthogonal projection matrix $P_\mathbb{V}$, this gives us an alternative model for the Grassmannian that is also widely used (notably in linear programming \cite{SSZ,Zhao} but also many other areas \cite{Chi,CHRSS,Conway,EG,Mattila,Nicolaescu}):
\begin{equation}\label{eq:HHT}
\Gr(k,n) \cong \{P \in \mathbb{R}^{n \times n} : P^\tp = P = P^2,\; \tr(P) = k\}.
\end{equation}
Note that $\rank(P) = \tr(P) = \dim(\mathbb{V})$ for orthogonal projection matrices. The reader is reminded that an \emph{orthogonal} projection matrix is not an orthogonal matrix --- the `orthogonal' describes the projection, not the matrix. To avoid confusion, we drop `orthogonal' from future descriptions --- all projection matrices in our article will be orthogonal projection matrices.

As demonstrated in \cite{HHT}, it is also possible to derive closed-form analytic expressions for various differential geometric objects and present various optimization algorithms in terms of the matrix coordinates in \eqref{eq:HHT}.  Nevertheless, the problem with the model \eqref{eq:HHT} is that algorithms based on projection matrices are almost always \emph{numerically unstable}, especially in comparison with algorithms based on orthogonal matrices. Roughly speaking an orthogonal matrix preserves (Euclidean) norms and therefore rounding errors do not get magnified through a sequence of orthogonal transformations  \cite[Section~3.4.4]{Demmel} and consequently algorithms based on orthogonal matrices tend to be numerically stable (details are more subtle, see \cite[pp.~124--166]{Wilkinson} and \cite{Higham}). Projection matrices not only do not preserve norms but are singular and  give notoriously unstable algorithms --- possibly the best known illustration of numerical instability \cite{Trefethen,Watkins} is one that contrasts Gram--Schmidt, which uses projection matrices, with Householder \textsc{qr}, which uses orthogonal matrices.\footnote{For example, computing the \textsc{qr} decomposition of a Hilbert matrix $A = [1/(i+j-1) ]_{i,j=1}^{15}$, we get $\lVert Q^* Q - I \rVert \approx 8.0 \times 10^{0}$ with Gram--Schmidt, $1.7 \times 10^{0}$ with modified Gram--Schmidt, $2.4 \times 10^{-15}$ with Householder \textsc{qr}.} In fact, the proper way to compute projections is to do so via a sequence of orthogonal matrices \cite[pp.~260--261]{Stewart}, as a straightforward computation is numerically unstable \cite[pp.~849--851]{CRT}.

The alternative \eqref{eq:EAS} is currently universally adopted for optimization over a Grassmannian. The main issue with the model \eqref{eq:EAS} is that a point on $ \Gr(k,n)$ is not a single matrix but an \emph{equivalence class} of uncountably many matrices. Equivalence classes are tricky to implement in numerical algorithms and standard algorithms in numerical linear algebra \cite{LAPACK}  do not work with equivalence classes of matrices. Indeed, any optimization algorithm \cite{Drei,EAS,JD,MS,SS,WY} that rely on the model \eqref{eq:EAS} side steps the issue by instead optimizing an $\O(k)$-invariant function $\tilde{f} : \V(k,n) \to \mathbb{R}$, i.e., where $\tilde{f}(VQ) = \tilde{f}(V)$ for all $Q \in \O(k)$. This practice makes it somewhat awkward to optimize over a Grassmannian: Given a function $f : \Gr(k,n) \to \mathbb{R}$ to be optimized, one needs to first lift it to another function $\tilde{f} : \V(k,n) \to \mathbb{R}$ and the choice of $\tilde{f}$ is necessarily ad hoc as there are uncountably many possibilities for $\tilde{f}$. Nevertheless this is presently the only viable option.

Numerical stability is of course relative, we will see in Section~\ref{sec:num}, for reasons explained therein, that the optimization algorithms in \cite{EAS} for the model \eqref{eq:EAS} are significantly less stable than those for our proposed model. In particular, the aforementioned loss-of-orthogonality remains very much an issue when one uses \eqref{eq:EAS} to represent a Grassmannian.

We would like to mention a noncompact analogue of \eqref{eq:EAS} that is popular in combinatorics \cite{AM,FP,GP,Karp,LF}:
\begin{equation}\label{eq:AMS}
\Gr(k,n) \cong \mathbb{R}^{n \times k}_k/\GL(k),
\end{equation}
with $\mathbb{R}^{n \times k}_k\coloneqq  \{A \in \mathbb{R}^{n \times k} : \rank(A) = k\}$. It has also been shown \cite{AMS} that one may obtain closed-form analytic expressions for differential geometric quantities with the model \eqref{eq:AMS} and so in principle one may use it for optimization purposes. Nevertheless, from the perspective of numerical algorithms, the model \eqref{eq:AMS} suffers from the same problem as \eqref{eq:HHT} --- by working with rank-$k$ matrices, i.e., whose condition number can be arbitrarily large, algorithms based on \eqref{eq:AMS} are inherently numerically unstable. In fact, since the model \eqref{eq:AMS} also represents points as equivalence classes, it has both shortcomings of \eqref{eq:EAS} and \eqref{eq:HHT} but neither of their good features. The natural redress of imposing orthogonal constraints on \eqref{eq:AMS} to get a well-conditioned representative for each equivalence class would just lead one back to the model \eqref{eq:EAS}.

Looking beyond optimization, we stress that each of the aforementioned models has its own (sometimes unique) strengths. For example, \eqref{eq:AMS} is the only model we know in which one may naturally define the positive Grassmannian \cite{FP}, an important construction in combinatorics \cite{Karp} and physics \cite{GP}. The model \eqref{eq:HHT} is indispensable in probability and statistics as probability measures \cite[Section~3.9]{Mattila} and probability densities \cite[Section~2.3.2]{Chi} on $\Gr(k,n)$ are invariably expressed in terms of projection matrices.

\subsection{Proposed model} We propose to use a model for the Grassmannian that combines the best features, suffers from none of the defects of  the aforementioned models, and, somewhat surprisingly, is also simpler:
\begin{equation}\label{eq:LLY}
\Gr(k,n) \cong \{Q \in \O(n) : Q^\tp = Q,\; \tr(Q) = 2k - n\}.
\end{equation}
This model, which represents $k$-dimensional subspace as a symmetric orthogonal matrix of trace $2k - n$, is known but obscure. It was mentioned in passing in \cite[p.~305]{Bhatia} and was used in \cite{JMS} to derive geodesics for the \emph{oriented Grassmannian}, a different but related manifold. Note that \eqref{eq:LLY} merely provides an expression for \emph{points}, our main contribution is to derive expressions for other differential geometric objects and operations, as well as the corresponding optimization algorithms, thereby fully realizing \eqref{eq:LLY} as a model for optimization.  A summary of these objects, operations, and algorithms is given in  Table~\ref{tab:formulas}. From a differential geometric perspective, Sections~\ref{sec:points}--\ref{sec:gradHess} may be regarded as an investigation into the embedded geometry of $\Gr(k,n)$ as a submanifold of $\O(n)$.
\begin{table}[h]
\footnotesize
\tabulinesep=0.5ex
\begin{tabu}{l|l}
\textsc{objects/operations} & \textsc{results}\\\hline\hline
point & Proposition~\ref{prop:points}\\\hline
change-of-coordinates & Proposition~\ref{prop:G1}, Proposition~\ref{prop:G2}, Proposition~\ref{prop:G3}, Proposition~\ref{prop:G4}\\\hline
tangent vector & Proposition~\ref{prop:tangent}, Proposition~\ref{prop:transtan}, Corollary~\ref{cor:tangent space}\\\hline
metric & Proposition~\ref{prop:metric}, Proposition~\ref{prop:isometric} \\\hline
normal vector & Proposition~\ref{prop:normal}, Corollary~\ref{cor:projection} \\\hline
curve & Proposition~\ref{prop:curve} \\\hline
geodesic & Theorem~\ref{thm:geo}, Proposition~\ref{prop:geo} \\\hline
geodesic distance & Corollary~\ref{cor:geodist}\\\hline
exponential map & Corollary~\ref{cor:exp} \\\hline
logarithmic map & Corollary~\ref{cor:log} \\\hline
parallel transport & Proposition~\ref{prop:pt}\\\hline
gradient & Proposition~\ref{prop:grad}, Corollary~\ref{cor:grad}\\\hline
Hessian & Proposition~\ref{prop:Hess}\\\hline
retraction and vector transport & Proposition~\ref{prop:Reig}, Proposition~\ref{prop:Rqr}, Proposition~\ref{prop:Rcay}\\\hline
steepest descent & Algorithm~\ref{alg:cay}, Algorithm~\ref{alg:sd}\\\hline
Newton method & Algorithm~\ref{alg:nt}\\\hline
conjugate gradient & Algorithm~\ref{alg:cg}\\\hline
quasi-Newton & Algorithm~\ref{alg:qn}
\end{tabu}
\bigskip
\caption{Guide to results.}
\label{tab:formulas}
\end{table}

The two key advantages of the model \eqref{eq:LLY} in computations are immediate:
\begin{enumerate}[\upshape (i)]
\item points on $\Gr(k,n)$ are represented as actual matrices, not equivalence classes;
\item the matrices involved are orthogonal, so numerical stability is preserved.
\end{enumerate}
The bonus with \eqref{eq:LLY}  is that the expressions and algorithms in Table~\ref{tab:formulas} are considerably simpler compared to those in \cite{AMS,EAS,HHT}. We will not need to solve quadratic eigenvalue problems, nor compute $\exp$/$\cos$/$\sin$/$\sinc$ of nonnormal matrices, nor even \textsc{evd} or \textsc{svd} except in cases when they can be trivially obtained. Aside from standard matrix arithmetic, our optimization algorithms require just two operations:
\begin{enumerate}[\upshape (i)]\setcounter{enumi}{2}
\item all differential geometric objects and operations can be computed with at most a \textsc{qr} decomposition and an exponentiation of a skew-symmetric matrix,
\[
\exp \biggl( \begin{bmatrix}
0 & B \\
-B^\tp & 0
\end{bmatrix} \biggr), \quad B \in \mathbb{R}^{k \times (n-k)},
\]
which may in turn be computed in time $O\bigl(nk(n-k)\bigr)$ with a specialized algorithm based on Strang splitting.
\end{enumerate}
The problem of computing matrix exponential has been thoroughly studied and there is a plethora of algorithms \cite{Higham2, MVL}, certainly more so than other transcendental matrix functions like cosine, sine, or sinc \cite{Higham2}. For normal matrices, matrix exponentiation is a well-conditioned problem --- the numerical issues described in \cite{MVL} only occur with nonnormal matrices. For us, $\begin{bsmallmatrix} 0 & B \\
\smash[b]{-B^\tp} & 0 \end{bsmallmatrix}$ is skew-symmetric and thus normal; in fact its exponential will always be an orthogonal matrix.

There are other algorithmic advantages afforded by \eqref{eq:LLY} that are difficult to explain without context and will be discussed alongside the algorithms in Section~\ref{sec:algo} and numerical results in Section~\ref{sec:num}.

\subsection{Nomenclatures and notations}

For easy reference, we will introduce names for the models \eqref{eq:EAS}--\eqref{eq:LLY} based on the type of matrices used as coordinates for points.
\begin{table}[h]
\footnotesize
\tabulinesep=0.75ex
\begin{tabu}{l|l|l}
\textsc{name} & \textsc{model} &  \textsc{coordinates for a point}\\\hline\hline
orthogonal model & $\O(n)/\bigl(\O(k) \times \O(n-k)\bigr)$ & equivalence class of $n \times n$ orthogonal matrices  $\lb V \rb$ \\\hline
Stiefel model & $\V(k,n)/\O(k)$ &  equivalence class of $n \times k$ orthonormal matrices $\lb Y \rb$  \\\hline
full-rank model & $\mathbb{R}^{n \times k}_k/\GL(k)$ &  equivalence class of $n \times k$  full-rank matrices $\lb A \rb$ \\\hline
projection model & $\{P \in \mathbb{R}^{n \times n} : P^\tp = P = P^2,\; \tr(P) = k\}$ & $n \times n$ orthogonal projection matrix $P$ \\\hline
involution model & $\{Q \in \O(n) : Q^\tp = Q,\; \tr(Q) = 2k - n\}$ & $n \times n$ symmetric involution matrix $Q$ \\\hline
\end{tabu}
\bigskip
\caption{Matrix manifold models for the Grassmannian $\Gr(k,n)$.}
\label{tab:models}
\end{table}

We note that there are actually two homogeneous space models for $\Gr(k,n)$ in \eqref{eq:EAS}, one as a quotient of $\O(n)$ and the other as a quotient of $\V(k,n)$. While they are used somewhat interchangeably in \cite{EAS}, we distinguish them in Table~\ref{tab:models} as their change-of-coordinates maps to the involution model are different (see Section~\ref{sec:points}).

The name \emph{involution model} is warranted for \eqref{eq:LLY} because for any $Q \in \mathbb{R}^{n \times n}$, any two of the following conditions clearly imply the third:
\[
Q^\tp Q = I, \qquad Q^\tp = Q, \qquad Q^2 = I.
\]
Thus a symmetric orthogonal matrix may also be viewed as a symmetric involution or an orthogonal involution matrix. We will need the eigendecomposition of a matrix in the involution model for \emph{all} of our subsequent calculations; for easy reference we state this as a lemma. Such an eigendecomposition is trivial to compute, requiring only a single \textsc{qr} decomposition (of the matrix $\frac12(I + Q)$; see Lemma~\ref{lem:qr}).
\begin{lemma}\label{lem:1eig}
Let $k = 1,\dots,n$ and $Q \in \mathbb{R}^{n \times n}$ be such that
\[
Q^\tp Q = I, \qquad Q^\tp = Q,\qquad \tr(Q) = 2k - n.
\]
Then $Q$ has an eigenvalue decomposition
\[
Q = V I_{k,n-k} V^\tp = [y_1, \dots, y_k, z_1, \dots, z_{n-k}] 
\begin{bmatrix}
1 & & & & &\\
& \ddots & & & &\\
& & 1 & & &\\
& & & -1 & &\\
& & & & \ddots &\\
& & & & & -1
\end{bmatrix}
\begin{bmatrix*}[l]
y_1^\tp \\
\vdots\\
y_k^\tp\\
z_1^\tp\\
\vdots\\
z_{n-k}^\tp
\end{bmatrix*},
\]
where $V \in \O(n)$ and $I_{k,n-k} \coloneqq \diag(I_k,-I_{n-k}) = \diag(1,\dots,1,-1,\dots,-1)$.
\end{lemma}
\begin{proof}
Existence of an eigendecomposition follows from the symmetry of $Q$. A symmetric involution has all eigenvalues  $\pm 1$ and the multiplicity of $1$  must be $k$ since  $\tr(Q) = 2k - n $. 
\end{proof}
Henceforth, for a matrix $Q$ in the involution model, we write
\begin{equation}\label{eq:1eig}
\begin{aligned}
Y_Q &\coloneqq [y_1,\dots,y_k] \in \V(k,n), \qquad
Z_Q \coloneqq [z_1,\dots,z_{n-k}] \in \V(n-k,n), \\
V_Q  &=  [Y_Q, Z_Q] = V \in \O(n)
\end{aligned}
\end{equation}
for its matrix of $1$-eigenvectors, its matrix of $-1$-eigenvectors, and its matrix of all eigenvectors respectively. While these matrices are not unique, the $1$-eigenspace and $-1$-eigenspace
\[
\im(Y_Q) = \spn \{y_1,\dots,y_k\} \in \Gr(k,n), \qquad \im(Z_Q) = \spn \{z_1,\dots,z_{n-k}\} \in \Gr(n-k,n)
\]
are uniquely determined by $Q$.

\section{Points and change-of-coordinates}\label{sec:points}

We begin by exhibiting a diffeomorphism to justify the involution model, showing that as \emph{smooth manifolds}, $\Gr(k,n)$ and $\{Q \in \O(n) : Q^\tp = Q,\; \tr(Q) = 2k - n\}$ are the same. In the next section, we will show that if we equip the latter with appropriate Riemannian metrics, then as \emph{Riemannian manifolds}, they are also the same, i.e., the diffeomorphism is an isometry. The practically minded may  simply take this as establishing a system of matrix coordinates for points on $\Gr(k,n)$.
\begin{proposition}[Points]\label{prop:points}
Let $k = 1,\dots,n$. Then the map
\begin{equation}
\begin{aligned}\label{pro:points:eqn:map}
\varphi : \Gr(k,n) &\to \{Q \in \O(n) : Q^\tp = Q,\; \tr(Q) = 2k - n\},\\
\varphi( \mathbb{W} ) &=   P_\mathbb{W} - P_{\mathbb{W}^\perp},
\end{aligned}
\end{equation} 
is a diffeomorphism with $\varphi^{-1} (Q) =\im(Y_Q)$ where $Y_Q \in \V(k,n)$ is as in \eqref{eq:1eig}.
\end{proposition}
\begin{proof}
One can check that $Q =P_\mathbb{W} - P_{\mathbb{W}^\perp}$ is symmetric, orthogonal, and has trace $2k -n$. So the map $\varphi$ is well-defined. If we write $\psi(Q) = \im(Y_Q)$, then $\varphi(\psi(Q)) = Q$ and $\psi (\varphi(\mathbb{W})) = \mathbb{W}$, so $\psi = \varphi^{-1}$. To see that $\varphi$ is smooth, we may choose any local coordinates, say, represent $\mathbb{W} \in \Gr(k,n)$ in terms of any orthonormal basis $W = [w_1,\dots, w_k] \in \V(k,n)$ and observe that
\[
\varphi(\mathbb{W}) = 2 WW^\tp - I,  
\] 
which is smooth. With a linear change-of-coordinates, we may assume that 
\[
W = \begin{bmatrix}
I_k \\
0 
\end{bmatrix}.
\]
The differential $(d \varphi)_\mathbb{W}$ is given by the (clearly invertible) linear map 
\[
(d \varphi)_{{\mathbb{W}} } \left(
\begin{bmatrix}
0 \\
X
\end{bmatrix}  \right) =  \begin{bmatrix}
I_k \\
0 
\end{bmatrix} \begin{bmatrix}
0 & X^\tp
\end{bmatrix}  + \begin{bmatrix}
0 \\
X
\end{bmatrix}  \begin{bmatrix}
I_k & 0 
\end{bmatrix} = \begin{bmatrix}
0 & X^\tp \\
X & 0
\end{bmatrix}
\]
for all $ X\in \mathbb{R}^{(n-k) \times k}$. So  $\varphi$ is a diffeomorphism.
\end{proof}

Since the manifolds in Table~\ref{tab:models} are all diffeomorphic to $\Gr(k,n)$, they are diffeomorphic to each other. Our next results are not intended to establish that they are diffeomorphic but to construct these diffeomorphisms and their inverses explicitly, so that we may switch to and from the other systems of coordinates easily. 

In the next proposition, $\lb V \rb = \Bigl\{V \begin{bsmallmatrix}Q_1 & 0\\0 & Q_2 \end{bsmallmatrix} : Q_1 \in \O(k), \; Q_2 \in \O(n-k)\Bigr\}$ denotes equivalence class in $\O(n)/\bigl(\O(k) \times \O(n-k)\bigr)$.
\begin{proposition}[Change-of-coordinates I]\label{prop:G1}
Let $k = 1,\dots,n$. Then
\begin{align*}
\varphi_1 : \O(n)/\bigl(\O(k) \times \O(n-k)\bigr) &\to \{Q \in \O(n) : Q^\tp = Q,\; \tr(Q) = 2k - n\},\\
\varphi_1(\lb V \rb) &= V^\tp I_{k,n-k} V
\end{align*}
is a diffeomorphism with $\varphi_1^{-1}  (Q) =  \lb V_Q \rb$ with $V_Q\in \O(n)$ as in \eqref{eq:1eig}.
\end{proposition}
\begin{proof}
Note that $Q = V_1 I_{k,n-k} V_1^\tp = V_2 I_{k,n-k} V_2^\tp$ iff
\[
V_2 = V_1 \begin{bmatrix}
Q_1 & 0 \\
0 & Q_2 
\end{bmatrix}
\]
for some $(Q_1,Q_2)\in \O(k) \times \O(n-k)$ iff $\lb V_1 \rb = \lb V_2 \rb$. Hence both $\varphi_1$ and $\varphi_1^{-1}$ are well-defined and are inverses of each other. Observe that $\varphi_1$ is induced from the map 
\[
\widetilde{\varphi}_1: \O(n) \to  \{Q \in \O(n) : Q^\tp = Q,\; \tr(Q) = 2k - n\},\quad \widetilde{\varphi}_1 (V) = V^\tp I_{k,n-k} V,
\]
which is a surjective submersion. The proof that $\varphi_1^{-1}$ is well-defined shows that the fibers of $\widetilde{\varphi}_1$ are exactly the $\O(k) \times \O(n-k)$-orbits in $\O(n)$. Hence  $\varphi_1$, as the composition of $\widetilde{\varphi}_1$ and the quotient map $\O(n) \to \O(n) / \bigl(\O(k) \times \O(n-k)\bigr)$, is a diffeomorphism.
\end{proof}

The next result explains the resemblance between the projection and involution models --- each is a scaled and translated copy of the other. The scaling and translation are judiciously chosen so that orthogonal projections become symmetric involutions, and this seemingly innocuous difference will have a significant impact on the numerical stability of Grassmannian optimization algorithms.
\begin{proposition}[Change-of-coordinates II]\label{prop:G2}
Let $k =1,\dots, n$. Then
\begin{align*}
\varphi_2 : \{P \in \mathbb{R}^{n \times n} : P^\tp = P = P^2,\; \tr(P) = k\} &\to \{Q \in \O(n) : Q^\tp = Q,\; \tr(Q) = 2k - n\},\\
\varphi_2(P) &= 2P - I
\end{align*}
is a diffeomorphism with $\varphi_2^{-1}  (Q) = \frac{1}{2}(I+Q)$.
\end{proposition}
\begin{proof}
Note that $2P - I = P - P^\perp$ where $P^\perp$ is the projection onto the orthogonal complement of $\im (P)$, so both $\varphi_2$ and $\varphi_2^{-1}$ are well-defined. They are clearly diffeomorphisms and are inverses to each other.
\end{proof}

In the next proposition, $\lb Y \rb = \{YQ : Q \in \O(k) \}$ denotes equivalence class in $\V(k,n)/\O(k)$.
\begin{proposition}[Change-of-coordinates III]\label{prop:G3}
Let $k = 1,\dots,n$. Then
\begin{align*}
\varphi_3 :  \V(k,n)/\O(k) &\to \{Q \in \O(n) : Q^\tp = Q,\; \tr(Q) = 2k - n\},\\
\varphi_3(\lb Y \rb) &= 2 Y Y^\tp - I
\end{align*}
is a diffeomorphism with $\varphi_3^{-1}  (Q) = \lb Y_Q \rb$ with $Y_Q \in \V(k,n)$ as in \eqref{eq:1eig}.
\end{proposition}
\begin{proof}
Given $\lb Y \rb \in \V(k,n)/\O(k)$, the matrix $Y Y^\tp$ is the projection matrix onto the $k$-dimensional subspace $\im(Y) \in \Gr(k,n)$. Hence $\varphi_3$ is a well-defined map by Proposition~\ref{prop:G2}. To show that its inverse is given by $\psi_3(Q) = \lb Y_Q \rb$, observe that any $Y\in \V(k,n)$ can be extended to a full orthogonal matrix $V \coloneqq [Y, Y^\perp]\in \O(n)$ and we have
\[
V^\tp Y = \begin{bmatrix}
I_k \\
0
\end{bmatrix},\qquad
Q = 2YY^\tp - I = V \begin{bmatrix}
2 I_k & 0 \\
0 & 0
\end{bmatrix} V^\tp - I = V I_{k,n-k} V^\tp.
\]
This implies that $\psi_3 \circ \varphi_3 (\lb Y \rb) = \lb Y_Q \rb = \lb Y \rb$. That $\varphi_3$ is a diffeomorphism follows from the same argument in the proof of Proposition~\ref{prop:points}.
\end{proof}

In the next proposition, $\lb A \rb = \{AX : X \in \GL(k) \}$ denotes equivalence class in $\mathbb{R}^{n \times k}_k/\GL(k)$. Also, we write $A = Y_A R_A$ for the \textsc{qr} factorization of $A \in \mathbb{R}^{n \times k}_k$, i.e., $Y_A\in \V(k,n)$ and $R_A \in \mathbb{R}^{k\times k}$ is upper triangular.
\begin{proposition}[Change-of-coordinates IV]\label{prop:G4}
Let $k = 1,\dots,n$. Then
\begin{align*}
\varphi_4 : \mathbb{R}^{n \times k}_k/\GL(k) &\to \{Q \in \O(n) : Q^\tp = Q,\; \tr(Q) = 2k - n\},\\
\varphi_4(\lb A \rb) &= 2Y_A Y_A^\tp-I
\end{align*}
is a diffeomorphism with $\varphi_4^{-1}  (Q) = \lb Y_Q \rb$ with $Y_Q$ is as in \eqref{eq:1eig}.
\end{proposition}
\begin{proof}
First observe that $\V(k,n) \subseteq \mathbb{R}^{n \times k}_k$ and the inclusion map $\V(k,n) \hookrightarrow \mathbb{R}^{n \times k}_k$ induces a diffeomorphism $ \V(k,n)/\O(k) \cong \mathbb{R}^{n \times k}_k/\GL(k)$ --- if we identify them, then $\varphi_4^{-1}$ becomes $\varphi_3^{-1}$ in Proposition~\ref{prop:G3} and is thus a diffeomorphism. It follows that $\varphi_4$ is a diffeomorphism. That the maps are inverses to each other follows from the same argument in the proof of Proposition~\ref{prop:G3}.
\end{proof}

The maps $\varphi,\varphi_1,\varphi_2,\varphi_3,\varphi_4$ allow one to transform an optimization problem formulated in terms of abstract $k$-dimensional subspaces or in terms of one of the first four models in Table~\ref{tab:models} into a mathematically (but not computationally) equivalent problem in terms of the involution model. Note that these are change-of-coordinate maps for \emph{points} --- they are good for translating expressions that involve only points on $\Gr(k,n)$. In particular, one cannot simply apply these maps to the analytic expressions for other differential geometric objects and operations in \cite{AMS, EAS, HHT} and obtain corresponding expressions for the involution model. Deriving these requires considerable effort and would take up the next three sections. 

Henceforth we will identify the Grassmannian with the involution model:
\[
\Gr(k,n) \coloneqq  \{Q \in \O(n) : Q^\tp = Q,\; \tr(Q) = 2k - n\},
\]
i.e., in the rest of our article, points on $\Gr(k,n)$ are symmetric orthogonal matrices of trace $2k-n$. With this, the well-known isomorphism
\begin{equation}\label{eq:iso}
\Gr(k,n)  \cong \Gr(n-k,n),
\end{equation}
which we will need later, is simply given by the map $Q \mapsto -Q$.

\section{Metric, tangents, and normals}\label{sec:tangent}

The simple observation in Lemma~\ref{lem:1eig} implies that a neighborhood of any point $Q \in \Gr(k,n)$ is just like a neighborhood of the special point $I_{k,n-k} = \diag(I_k,-I_{n-k}) \in \Gr(k,n)$. Consequently, objects like tangent spaces and curves at $Q$ can be determined by simply determining them at $I_{k,n-k}$. Although $\Gr(k,n)$ is not a Lie group, the involution model, which models it as a linear section of $\O(n)$, allows certain characteristics of a Lie group to be retained. Here $I_{k,n-k}$ has a role similar to that of the identity element in a Lie group.

We will provide three different expressions for vectors in the \emph{tangent space} $\T_Q \Gr(k,n)$ at a point $Q\in \Gr(k,n)$: an implicit form \eqref{eq:tan1} as traceless symmetric matrices that anticommutes with $Q$ and two explicit forms \eqref{eq:tan2}, \eqref{eq:tan3} parameterized by $k \times (n-k)$ matrices. Recall from Lemma~\ref{lem:1eig} that any  $Q  \in \Gr(k,n)$ has an eigendecomposition of the form $Q = V I_{k,n-k} V^\tp$ for some $V\in \O(n)$.
\begin{proposition}[Tangent space I]\label{prop:tangent}
Let  $Q  \in \Gr(k,n)$ with eigendecomposition $Q = V I_{k,n-k} V^\tp$.
The tangent space of $\Gr(k,n)$ at $Q$ is given by
\begin{align}
\T_Q \Gr(k,n) &= \left\lbrace X\in \mathbb{R}^{n\times n}: X^\tp = X,\; X Q +QX = 0,\; \tr(X) = 0 \right\rbrace \label{eq:tan1}
 \\
& = \left\lbrace
V \begin{bmatrix}
0 & B \\
B^\tp & 0
\end{bmatrix}V^\tp  \in \mathbb{R}^{n \times n} : B \in \mathbb{R}^{k\times (n-k)}
\right\rbrace \label{eq:tan2} \\
& = \left\lbrace 
QV \begin{bmatrix}
0 & B \\
-B^\tp & 0 
\end{bmatrix} V^\tp  \in \mathbb{R}^{n \times n}: B\in \mathbb{R}^{k \times (n-k)} 
\right\rbrace. \label{eq:tan3}
\end{align}
\end{proposition}
\begin{proof}
By definition, a curve $\gamma$ in $\Gr(k,n)$ passing through $Q$ satisfies
\[
\gamma(t)^\tp - \gamma(t) = 0,\quad  \gamma(t)^\tp \gamma(t) = I_n, \quad \tr(\gamma(t)) = 2k - n,\quad t\in (-\varepsilon, \varepsilon),
\]
together with the initial condition $\gamma(0) = Q$. Differentiating these equations at $t = 0$, we get
\[
\dot{\gamma}(0)^\tp -  \dot{\gamma}(0) = 0, \quad \dot{\gamma}(0)^\tp Q + Q^\tp \dot{\gamma}(0) = 0, \quad \tr(\dot{\gamma}(0)) = 0,
\]
from which \eqref{eq:tan1} follows. Now take $X\in \T_Q \Gr(k,n)$. By \eqref{eq:tan1}, $ V^\tp X V I_{k,n-k}=V^\tp (X Q)V $ is skew-symmetric and $V^\tp X V$ is symmetric. Partition
\[
V^\tp X V = \begin{bmatrix}
A & B \\
B^\tp & C
\end{bmatrix},\qquad A \in \mathbb{R}^{k \times k},\; B \in \mathbb{R}^{k\times (n-k)},\; C \in \mathbb{R}^{(n-k) \times (n-k)}.
\]
Note that $A$ and $C$ are symmetric matrices since $X$ is. So if
\[
 V^\tp X V I_{k,n-k}=\begin{bmatrix}
A & B \\
B^\tp & C
\end{bmatrix}\begin{bmatrix}
I & 0 \\
0 & -I
\end{bmatrix}=\begin{bmatrix}
A & -B \\
B^\tp & -C
\end{bmatrix}
\]
is skew-symmetric, then we must have $A =0$ and $C = 0$ and we obtain  \eqref{eq:tan2}. Since $Q = V I_{k,n-k} V^\tp$ and $Q = Q^\tp$, \eqref{eq:tan3} follows from \eqref{eq:tan2} by writing $V = Q V I_{k,n-k}$.
\end{proof}
The implicit form in \eqref{eq:tan1} is inconvenient in algorithms. Of the two explicit forms \eqref{eq:tan2} and \eqref{eq:tan3}, the description in \eqref{eq:tan2} is evidently more economical, involving only $V$, as opposed to both $Q$ and $V$ as in \eqref{eq:tan3}. Henceforth, \eqref{eq:tan2} will be our preferred choice and we will assume that a tangent vector at $Q \in \Gr(k,n)$ always takes the form
\begin{equation}\label{eq:tanvec}
X = 
V \begin{bmatrix}
0 & B \\
B^\tp & 0
\end{bmatrix}V^\tp,
\end{equation}
for some $B \in \mathbb{R}^{k \times (n-k)}$.
This description appears to depend on the eigenbasis $V$, which is not unique, as $Q$ has many repeated eigenvalues. The next proposition, which relates two representations of the same tangent vector with respect to two different $V$'s, guarantees that the tangent space obtained will nonetheless be the same regardless of the choice of $V$.
\begin{proposition}[Tangent vectors]\label{prop:transtan}
If $V_1 I_{k,n-k} V_1^\tp = Q= V_2 I_{k,n-k} V_2^\tp$, then any $X \in \T_Q \Gr(k,n)$ can be written as
\[
X = V_2 \begin{bmatrix}
0 & B \\
B^\tp & 0
\end{bmatrix} V_2^\tp =V_1   \begin{bmatrix}
0 & Q_1 B Q_2^\tp \\
Q_2 B^\tp Q_1^\tp & 0
\end{bmatrix} V_1^\tp,
\]
for some $Q_1 \in \O(k)$ and $Q_2\in \O(n-k)$ such that
\begin{equation}\label{eq:transtan}
V_2 = V_1 \begin{bmatrix}
Q_1 & 0 \\
0 & Q_2
\end{bmatrix}.
\end{equation} 
\end{proposition}
\begin{proof}
This is a consequence of the fact that  $V_1 I_{k,n-k} V_1^\tp = Q= V_2 I_{k,n-k} V_2^\tp$ iff there exist $Q_1 \in \O(k)$ and $Q_2\in \O(n-k)$ such that \eqref{eq:transtan} holds.
\end{proof}

Another consequence of using \eqref{eq:tan2} is that the tangent space at any point $Q$ is a copy of the tangent space at $I_{k,n-k}$, conjugated by any eigenbasis $V$ of $Q$; by Proposition~\ref{prop:transtan}, this is independent of the choice of $V$.
\begin{corollary}[Tangent space II]\label{cor:tangent space}
The tangent space at $I_{k,n-k}$ is
\[
\T_{I_{k,n-k}} \Gr(k,n) = \left\lbrace
 \begin{bmatrix}
0 & B \\
B^\tp & 0
\end{bmatrix}:B \in \mathbb{R}^{k\times {(n-k)}}
\right\rbrace.
\]
For any $Q \in \Gr(k,n)$ with eigendecomposition $Q = V I_{k,n-k} V^\tp$,
\[
\T_Q \Gr(k,n) = V \bigl( \T_{I_{k,n-k}} \Gr(k,n) \bigr) V^\tp.
\]
\end{corollary}

With the tangent spaces  characterized, we may now define an inner product $\langle \cdot, \cdot\rangle_Q$ on each $\T_Q\Gr(k,n)$ that varies smoothly over all $Q \in \Gr(k,n)$, i.e., a \emph{Riemannian metric}. With the involution model, $\Gr(k,n)$ is a submanifold of $\O(n)$ and there is a natural choice, namely, the  Riemannian metric inherited from that on $\O(n)$.
\begin{proposition}[Riemannian metric]\label{prop:metric}
Let $Q\in \Gr(k,n)$ with $Q = V I_{k,n-k} V^\tp$ and
\[
X = 
V \begin{bmatrix}
0 & B \\
B^\tp & 0
\end{bmatrix}V^\tp
,\quad Y = 
V \begin{bmatrix}
0 & C \\
C^\tp & 0
\end{bmatrix}V^\tp
 \in \T_Q \Gr(k,n).
\]
Then
\begin{equation}\label{eq:metric}
\langle X, Y \rangle_Q \coloneqq \tr(XY) =2\tr(B^\tp C) 
\end{equation}
defines a Riemannian metric.
The corresponding Riemannian norm is
\begin{equation}\label{eq:norm}
\lVert X \rVert_Q \coloneqq \sqrt{\langle X, X \rangle}_Q= \lVert X \rVert_\F =\sqrt{2} \lVert B \rVert_\F.
\end{equation}
\end{proposition}
The Riemannian metric in \eqref{eq:metric} is induced by the  unique (up to a positive constant multiple) bi-invariant Riemannian metric on $\O(n)$:
\[
g_Q (X, Y) \coloneqq \tr(X^\tp Y),\quad Q\in \O(n), \quad X,Y\in \T_Q \O(n).
\]
Here bi-invariance may be taken to mean
\[
g_{V_1 Q V^\tp_2} (V_1 X V^\tp_2, V_1Y V^\tp_2) = g_Q (X, Y)
\]
for all $Q, V_1, V_2 \in \O(n)$ and $X,Y\in \T_Q \O(n)$. 

There are also natural Riemannian metrics \cite{AMS,EAS,HHT} on the other four models in Table~\ref{tab:models} but they differ from each other by a constant. As such, it is not possible for us to choose our metric \eqref{eq:metric} so that the diffeomorphisms in Propositions~\ref{prop:G1}--\ref{prop:G4} are all isometry but we do have the next best thing.
\begin{proposition}[Isometry]\label{prop:isometric}
All models in Table~\ref{tab:models} are, up to a constant factor, isometric as Riemannian manifolds.
\end{proposition}
\begin{proof}
We verify that the diffeomorphism $\varphi_1$ in Proposition~\ref{prop:G1} gives an isometry between the orthogonal model and the involution model up a constant factor of $8$.  A tangent vector \cite[Equation~2.30]{EAS} at a point $\lb V \rb \in \O(n)/\bigl(\O(k) \times \O(n-k)\bigr)$ takes the form
\[
V \begin{bmatrix}
0 & B \\
-B^\tp & 0
\end{bmatrix} \in \T_{\lb V \rb}  \O(n)/\bigl(\O(k) \times \O(n-k)\bigr), \quad B \in \mathbb{R}^{k \times (n-k)};
\]
and the Riemannian metric \cite[Equation~2.31]{EAS} on $\O(n)/\bigl(\O(k) \times \O(n-k)\bigr)$ is given by
\[
g_{\lb V \rb} \biggl(V \begin{bmatrix}
0 & B_1\\
-B_1^\tp & 0
\end{bmatrix}, V \begin{bmatrix}
0 & B_2\\
-B_2^\tp & 0
\end{bmatrix} \biggr) = \tr(B_1^\tp B_2).
\]
At $I_n$, the differential can be computed by
\[
(d \varphi_1)_{\lb I_n \rb} \biggl(I_n \begin{bmatrix}
0 & B\\
-B^\tp & 0
\end{bmatrix}\biggr) =2 I_{k,n-k} \begin{bmatrix}
0 & B\\
-B^\tp & 0
\end{bmatrix} =2  \begin{bmatrix}
0 & B\\
B^\tp & 0
\end{bmatrix}.
\]
Since both $g$ and $\langle \cdot, \cdot \rangle$ are invariant under left multiplication by $\O(n)$, we have
\[
\biggl\langle (d \varphi_1)_{\lb V \rb} \biggl(V \begin{bmatrix}
0 & B_1\\
-B_1^\tp & 0
\end{bmatrix}\biggr), (d \varphi_1)_{\lb V \rb} \biggl(V \begin{bmatrix}
0 & B_1\\
-B_1^\tp & 0
\end{bmatrix}\biggr) \biggr\rangle_{\varphi_1(\lb V \rb)} = 8\tr(B_1^\tp B_2).
\]
The proofs for $\varphi_2,\varphi_3,\varphi_4$ are similar and thus omitted.
\end{proof}
As the above proof shows, the diffeomorphism $\varphi_1$ may be easily made an isometry of the orthogonal and involution models by simply changing our metric in \eqref{eq:metric} to ``$\langle X, Y \rangle_Q \coloneqq \frac{1}{8} \tr(XY)$.'' Had we wanted to make $\varphi_2$ into an isometry of the projection and involution models, we would have to choose ``$\langle X, Y \rangle_Q \coloneqq \frac{1}{2} \tr(XY)$'' instead. We see no reason to favor any single existing model and we stick to our choice of metric in \eqref{eq:metric}.

In the involution model, $\Gr(k,n) \subseteq \O(n)$ as a smoothly embedded submanifold and every point  $Q\in \Gr(k,n)$ has a \emph{normal space} $\N_Q \Gr(k,n)$.  We will next determine the expressions for normal vectors.
\begin{proposition}[Normal space]\label{prop:normal}
Let  $Q  \in \Gr(k,n)$ with  $Q = V I_{k,n-k} V^\tp$. The normal space of $\Gr(k,n)$ at $Q$ is given by
\[
\N_{Q} \Gr(k,n) = \biggl\lbrace 
V \begin{bmatrix}
\Lambda_1 & 0 \\
0 & \Lambda_2 
\end{bmatrix} V^\tp \in \mathbb{R}^{n \times n} :
\begin{aligned}
\Lambda_1&\in \mathbb{R}^{k \times k}, &\Lambda_2 &\in \mathbb{R}^{(n-k) \times (n-k)}\\
\smash[t]{\Lambda_1^\tp} &= - \Lambda_1, & \smash[t]{\Lambda_2^\tp} &=- \Lambda_2
\end{aligned}
\biggr\rbrace.
\]
\end{proposition}
\begin{proof}
The tangent space of a point $Q \in \O(n)$ is given by
\[
\T_Q \O(n) = \{ Q \Lambda \in \mathbb{R}^{n \times n} : \Lambda^\tp =- \Lambda \}.
\]
A tangent vector $Q \Lambda  \in \T_Q \O(n)$ is normal to $\Gr(k,n)$ at $Q$ iff
\[
0 = \langle X, Q \Lambda \rangle_Q =\tr(X^\tp Q \Lambda ),
\]
for all $X\in \T_Q\Gr(k,n)$. By \eqref{eq:tanvec}, $X = V\begin{bsmallmatrix}
0 & B \\
B^\tp & 0
\end{bsmallmatrix}V^\tp$ where $Q = V I_{k,n-k} V^\tp$. Thus
\begin{equation}\label{eq:nor2}
\tr \left( V^\tp \Lambda V \begin{bmatrix}
0 & -B \\
B^\tp & 0
\end{bmatrix}  \right) = 0
\end{equation}
for all $B\in \mathbb{R}^{k \times (n-k)}$. Since \eqref{eq:nor2} must hold for all $B\in \mathbb{R}^{k \times (n-k)}$, we must have
\begin{equation}\label{eq:nor3}
\Lambda  = 
V \begin{bmatrix}
\Lambda_1 & 0 \\
0 &  \Lambda_2
\end{bmatrix} V^\tp, 
\end{equation}
for some skew-symmetric matrices $\Lambda_1 \in \mathbb{R}^{k\times k}$, $\Lambda_2 \in \mathbb{R}^{(n-k) \times (n-k)}$, and therefore,
\[
Q \Lambda = V I_{k,n-k} V^\tp \Lambda = V \begin{bmatrix}
\Lambda_1 & 0 \\
0 & -\Lambda_2
\end{bmatrix} V^\tp.
\]
Conversely, any $\Lambda$ of the form in \eqref{eq:nor3} must satisfy \eqref{eq:nor2}.
\end{proof}

Propositions~\ref{prop:tangent} and \ref{prop:normal} allow us to explicitly decompose the tangent space of  $\O(n)$ at a point $Q \in \Gr(k,n)$ into
\begin{align*}
\T_Q \O(n) = \T_Q \Gr(k,n) &\oplus \N_Q \Gr(k,n),\\
Q \Lambda =
QV \begin{bmatrix}
0 & B \\
-B^\tp & 0 
\end{bmatrix} V^\tp &+ V \begin{bmatrix}
\Lambda_1 & 0 \\
0 & \Lambda_2 
\end{bmatrix} V^\tp .
\end{align*}
For later purposes, it will be useful to give explicit expressions for the two projection maps.
\begin{corollary}[Projection maps]\label{cor:projection}
Let  $Q  \in \Gr(k,n)$ with $Q = V I_{k,n-k} V^\tp$ and
\[
\proj^{\T}_Q : \T_Q \O(n) \to \T_Q \Gr(k,n), \qquad  \proj^{\N}_Q : \T_Q \O(n) \to \N_Q \Gr(k,n)
\]
be the projection maps onto the tangent and normal spaces of $\Gr(k,n)$ respectively. Then
\begin{equation}\label{eq:proj}
\begin{aligned}
\proj^{\T}_Q (Q \Lambda) &= \frac{1}{2} ( Q \Lambda  -  \Lambda Q) = \frac{1}{2} V ( S + S^\tp) V^\tp,\\
\proj^{\N}_Q (Q \Lambda) &= \frac{1}{2} ( Q  \Lambda + \Lambda Q) = \frac{1}{2} V ( S - S^\tp) V^\tp,
\end{aligned}
\end{equation}
for any decomposition $Q \Lambda = V S V^\tp$ where $S \in \mathbb{R}^{n \times n}$ is such that $I_{k,n-k} S$ is skew-symmetric
\end{corollary}
\begin{proof}
We see from Propositions~\ref{prop:tangent} and \ref{prop:normal} that the maps are well defined, i.e.,  $\frac{1}{2} ( Q \Lambda  -  \Lambda Q)  \in \T_Q \Gr(k,n)$ and $\frac{1}{2} ( Q  \Lambda + \Lambda Q) \in \N_Q \Gr(k,n)$, and the images are orthogonal as
\[
\langle  Q \Lambda  -  \Lambda Q, Q  \Lambda + \Lambda Q \rangle_{Q} = 0.
\]
The alternative expressions follow from taking  $S = I_{k,n-k} V^\tp \Lambda V $.
\end{proof}

\section{Exponential map, geodesic, and parallel transport}\label{sec:exp}

An explicit and easily computable formula for a geodesic curve is indispensable in most Riemannian optimization algorithms. By Lemma~\ref{lem:1eig}, any $Q\in \Gr(k,n)$ can be eigendecomposed as $V I_{k,n-k} V^\tp$ for some $V \in \O(n)$. So a curve $\gamma$ in $\Gr(k,n)$ takes the form
\begin{equation}\label{eq:curve1}
\gamma(t) = V(t) I_{k,n-k} V(t)^\tp,
\end{equation}
with $V(t)$ a curve in $\O(n)$ that can in turn be written as
\begin{equation}\label{eq:curve2}
V(t) =V \exp(\Lambda(t)),
\end{equation}
where $\Lambda(t)$ is a curve in the space of $n\times n$ skew-symmetric matrices, $\Lambda(0) = 0$, and $V(0) = V$. We will show in Proposition~\ref{prop:curve} that in the involution model the curve $\Lambda(t)$ takes a particularly simple form. We first prove a useful lemma using the \textsc{cs} decomposition \cite{GNS, Stewart2}.
\begin{lemma}\label{lem:skew}
Let $\Lambda \in \mathbb{R}^{n\times n}$ be skew-symmetric. Then there exist $B\in \mathbb{R}^{k \times (n-k)}$ and two skew-symmetric matrices $ \Lambda_1 \in \mathbb{R}^{k\times k}$, $ \Lambda_2 \in \mathbb{R}^{(n-k) \times (n-k)}$ such that 
\begin{equation}\label{eq:decomp}
\exp(\Lambda)  = \exp \biggl( \begin{bmatrix}
0 & B \\
-B^\tp & 0
\end{bmatrix} \biggr) \exp \biggl( \begin{bmatrix}
\Lambda_1 & 0 \\
0 & \Lambda_2
\end{bmatrix}  \biggr).
\end{equation}
\end{lemma}
\begin{proof}
By \eqref{eq:iso}, we may assume $k \le n/2$. Let the \textsc{cs} decomposition of $Q \coloneqq \exp (\Lambda) \in \O(n)$ be
\[
Q = \begin{bmatrix}
U & 0 \\
0 & V
\end{bmatrix} \begin{bmatrix}
\cos \Theta & \sin \Theta & 0  \\
-\sin \Theta & \cos \Theta & 0 \\
0 & 0 & I_{n-2k}
\end{bmatrix}  \begin{bmatrix}
W & 0 \\
0 & Z
\end{bmatrix}^\tp,
\]
where $U, W\in \O(k)$, $V,Z\in \O(n-k)$, and $\Theta = \diag(\theta_1,\dots, \theta_k)$ with $\theta_i \in [0,\pi/2]$, $i=1,\dots,k$. We may write 
\begin{align*}
\begin{bmatrix}
U & 0 \\
0 & V
\end{bmatrix}
\begin{bmatrix}
\cos \Theta & \sin \Theta & 0  \\
-\sin \Theta & \cos \Theta & 0 \\
0 & 0 & I_{n-2k}
\end{bmatrix}   &=
\exp \biggl( 
\begin{bmatrix}
U & 0 \\
0 & V
\end{bmatrix}
\begin{bmatrix}
0 & \Theta & 0 \\
-\Theta & 0 & 0 \\
0 & 0 & 0
\end{bmatrix}
\begin{bmatrix}
U & 0 \\
0 & V
\end{bmatrix}^\tp
\biggr) 
\begin{bmatrix}
U & 0 \\
0 & V
\end{bmatrix}
\\
&= \exp \biggl( \begin{bmatrix}
0 & B \\
-B^\tp & 0
\end{bmatrix} \biggr) 
\begin{bmatrix}
U & 0 \\
0 & V
\end{bmatrix},
\end{align*}
where $B \coloneqq U [\Theta,0] V^\tp \in \mathbb{R}^{k \times (n-k)}$ with $0 \in \mathbb{R}^{k \times (n-2k)}$. Finally, let $\Lambda_1,\Lambda_2$ be skew symmetric matrices such that $\exp(\Lambda_1) = UW^\tp$ and $\exp(\Lambda_2) = VZ^\tp$.
\end{proof}

\begin{proposition}[Curve]\label{prop:curve}
Let $Q\in \Gr(k,n)$ with eigendecomposition $Q = V I_{k,n-k} V^\tp$. Then a curve $\gamma(t)$  in $\Gr(k,n)$ through $Q$ may be expressed as
\begin{equation}\label{eq:gencurve}
\gamma(t) = V \exp \left( \begin{bmatrix}
0 & B(t) \\
-B(t)^\tp & 0
\end{bmatrix} \right) I_{k,n-k} \exp \left( \begin{bmatrix}
0 & -B(t) \\
B(t)^\tp & 0
\end{bmatrix} \right) V^\tp
\end{equation}
for some curve $B(t)$ in $\mathbb{R}^{k \times (n-k)}$ through the zero matrix.
\end{proposition}
\begin{proof}
By \eqref{eq:curve1} and \eqref{eq:curve2}, we have
\[
\gamma(t) = V \exp \bigl(\Lambda(t)\bigr) I_{k,n-k} \exp\bigl(-\Lambda(t)\bigr) V^\tp.
\]
By Lemma~\ref{lem:skew}, we may write 
\[
\exp \bigl(\Lambda(t)\bigr) =\exp \left( \begin{bmatrix}
0 & B(t) \\
-B(t)^\tp & 0
\end{bmatrix} \right) \exp \left( \begin{bmatrix}
\Lambda_1(t) & 0 \\
0 & \Lambda_2(t)
\end{bmatrix}  \right),
\] 
which gives the desired parametrization in \eqref{eq:gencurve}.
\end{proof}

Proposition~\ref{prop:curve} yields another way to obtain the expression for tangent vectors in \eqref{eq:tanvec}. Differentiating the curve in \eqref{eq:gencurve} at $t = 0$, we get
\[
\dot{\gamma}(0) =  V \left(  \begin{bmatrix}
0 & -2\dot{B}(0) \\
-2\dot{B}(0)^\tp & 0
\end{bmatrix} I_{k,n-k}  \right) V^\tp \in \T_{Q} \Gr(k,n).
\]
Choosing $B(t)$ to be any curve in $\mathbb{R}^{k\times (n-k)}$ with $B(0) = 0$ and $\dot{B}(0) = -B/2$, we obtain \eqref{eq:tanvec}.

The key ingredient in most manifold optimization algorithms is the geodesic at a point in a direction. In \cite{EAS}, the discussion regarding geodesics on the Grassmannian is brief: Essentially, it says that because a geodesic on the Stiefel manifold $\V(k,n)$ takes the form $Q \exp (t \Lambda)$, a geodesic on the Grassmannian $\V(k,n)/\O(k)$ takes the form $\lb Q \exp (t \Lambda) \rb$. It is hard to be more specific when one uses the Stiefel model. On the other hand, when we use the involution model, the expression \eqref{eq:geo} in the next theorem describes a geodesic precisely, and any point on $\gamma$ can be evaluated with a single \textsc{qr} decomposition (to obtain $V$, see Section~\ref{sec:subroutines}) and a single matrix exponentiation (the two exponents are transposes of each other).
\begin{theorem}[Geodesics I]\label{thm:geo}
Let $Q\in \Gr(k,n)$ and $X \in \T_Q \Gr(k,n)$ with
\begin{equation}\label{eq:geotan}
Q = V I_{k,n-k} V^\tp, \qquad  X = V\begin{bmatrix}
0 & B \\
B^\tp & 0
\end{bmatrix} V^\tp.
\end{equation} 
The geodesic $\gamma$ emanating from $Q$ in the direction $X$ is given by
\begin{equation}\label{eq:geo}
\gamma(t) = V \exp \left( \frac{t}{2} \begin{bmatrix}
0 & -B \\
B^\tp & 0
\end{bmatrix} \right) I_{k,n-k} \exp \left( \frac{t}{2}\begin{bmatrix}
0 & B \\
-B^\tp & 0
\end{bmatrix} \right) V^\tp.
\end{equation}
The differential equation for $\gamma$ is
\begin{equation}\label{eq:geode}
\gamma(t)^\tp \ddot{\gamma}(t) -  \ddot{\gamma}(t)^\tp \gamma(t) = 0,\qquad \gamma(0) = Q,\qquad \dot{\gamma}(0) = X.
\end{equation}
\end{theorem}
\begin{proof}
By Proposition~\ref{prop:curve}, any curve through $Q$ must take the form
\[
\gamma(t) = V \exp \left( \begin{bmatrix}
0 & B(t) \\
-B(t)^\tp & 0
\end{bmatrix} \right) I_{k,n-k} \exp \left( \begin{bmatrix}
0 & -B(t) \\
B(t)^\tp & 0
\end{bmatrix} \right) V^\tp,
\]
where $B(0) = 0$. Since $\gamma$ is in the direction $X$, we have that $\dot{\gamma}(0) = X$, and thus $\dot{B}(0) =-B/2$.
It remains to employ the fact that as a geodesic, $\gamma$ is a critical curve of the length functional
\[
L(\gamma) \coloneqq \int_0^1 \lVert \dot{\gamma}(t) \rVert_{\gamma(t)} \, dt
\]
where the Riemannian norm is as in \eqref{eq:norm}. 
Let $\varepsilon >0$. Consider a variation of $\gamma(t)$ with respect to a $C^1$-curve $C(t)$ in $\mathbb{R}^{k \times (n-k)}$:
\[
\begin{adjustbox}{width=\textwidth,totalheight=\textheight,keepaspectratio}
$
\gamma_\varepsilon(t) = V \exp \left( \begin{bmatrix}
0 & B(t) + \varepsilon C(t) \\
-B(t)^\tp- \varepsilon C(t)^\tp & 0
\end{bmatrix} \right) I_{k,n-k} \exp \left( \begin{bmatrix}
0 & -B(t) - \varepsilon C(t) \\
B(t)^\tp + \varepsilon C(t)^\tp & 0
\end{bmatrix} \right) V^\tp.
$
\end{adjustbox}
\]
We require $C(0) = C(1) = 0$ so that $\gamma_\varepsilon $ is a variation of $\gamma$ with fixed end points. The tangent vector of $\gamma_\varepsilon$ at time $t$ is given by 
\[
\begin{adjustbox}{width=\textwidth,totalheight=\textheight,keepaspectratio}
$
 V \exp \left( \begin{bmatrix}
0 & B(t) + \varepsilon C(t) \\
-B(t)^\tp- \varepsilon C(t)^\tp & 0
\end{bmatrix} \right) \left( -2  \begin{bmatrix}
0 & \dot{B}(t) + \varepsilon \dot{C}(t) \\
\dot{B}(t) + \varepsilon \dot{C}(t)^\tp & 0
\end{bmatrix}  \right)     \exp \left( \begin{bmatrix}
0 & -B(t) - \varepsilon C(t) \\
B(t)^\tp + \varepsilon C(t)^\tp & 0
\end{bmatrix} \right) V^\tp
$
\end{adjustbox}
\]
and so $\lVert \dot{\gamma_\varepsilon}(t) \rVert_
{\gamma(t)} =  2\sqrt{2} \lVert \dot{B}(t) + \varepsilon \dot{C}(t) \rVert_\F$ where $\lVert\, \cdot\, \rVert_\F$ denotes Frobenius norm. Hence,
\[
0 = \frac{d}{d \varepsilon} L \bigl( \gamma_\varepsilon (t) \bigr) \Bigr\rvert_{\varepsilon = 0} =2\sqrt{2}  \int_0^1 \frac{\tr\bigl(\dot{B}(t)^\tp \dot{C}(t)\bigr)}{\lVert \dot{B}(t) \rVert_\F} \, dt.
\]
As $\gamma (t)$ is a geodesic, $\lVert \dot{\gamma} (t) \rVert_{\gamma(t)} $ and thus  $\lVert \dot{B} (t) \rVert_\F$ must be a constant $K > 0$. Therefore, we have 
\[
0 = \frac{1}{K} \int_0^1 \tr\bigl( \dot{B}(t)^\tp \dot{C}(t)\bigr) \, dt =- \frac{1}{K} \int_0^1 \tr\bigl(\ddot{B}(t)^\tp  C(t) \bigr) \, dt,
\]
implying that $\ddot{B}(t) = 0$ and thus $B(t) = t\dot{B}(0)= -tB/2$. Lastly, since 
\begin{equation}\label{eq:dotgamma}
\begin{adjustbox}{width=0.9\textwidth,totalheight=\textheight,keepaspectratio}
$
\begin{aligned}
\dot{\gamma}(t) &=  V \exp \left( \begin{bmatrix}
0 & B(t)  \\
-B(t)^\tp & 0
\end{bmatrix} \right) \left( -2  \begin{bmatrix}
0 & \dot{B }(t) \\
\dot{B }(t)^\tp & 0
\end{bmatrix}  \right)     \exp \left( \begin{bmatrix}
0 & -B(t)  \\
B(t)^\tp  & 0
\end{bmatrix} \right) V^\tp, \\
\ddot{\gamma}(t) &= V \exp \left( \begin{bmatrix}
0 & B(t) \\
-B(t)^\tp & 0
\end{bmatrix} \right) \left( 
-4 \begin{bmatrix}
\dot{B}(t)\dot{B}(t)^\tp & 0 \\
0 & -\dot{B}(t)^\tp \dot{B}(t)    
\end{bmatrix}
-2  \begin{bmatrix}
0 & \ddot{B }(t) \\
\ddot{B }(t)^\tp & 0
\end{bmatrix}  \right) 
 \exp \left( \begin{bmatrix}
0 & -B(t)  \\
B(t)^\tp  & 0
\end{bmatrix} \right) V^\tp,
\end{aligned}
$
\end{adjustbox}
\end{equation}
and the differential equation for a geodesic curve $\gamma$ is 
\[
\proj^{\T}_{\gamma(t)} (\ddot{\gamma})  = 0,\qquad \gamma(0) = Q,\qquad \dot{\gamma}(0) = X,
\]
we obtain \eqref{eq:geode} from the expression for tangent projection in \eqref{eq:proj}.
\end{proof}
Theorem~\ref{thm:geo} also gives the \emph{exponential map} of $X$.
\begin{corollary}[Exponential map]\label{cor:exp}
Let $Q\in \Gr(k,n)$ and $X \in \T_Q \Gr(k,n)$ be as in \eqref{eq:geotan}.
Then
\begin{equation}\label{eq:exp}
\exp_Q(X) \coloneqq \gamma(1) = V \exp \left(\frac{1}{2} \begin{bmatrix}
0 & -B \\
B^\tp & 0
\end{bmatrix} \right) I_{k,n-k} \exp \left(\frac{1}{2} \begin{bmatrix}
0 & B \\
-B^\tp & 0
\end{bmatrix} \right) V^\tp.
\end{equation}
The length of the geodesic segment from $\gamma(0)=0$ to $\gamma(1)=\exp_Q(X) $ is 
\begin{equation}\label{eq:arc}
L(\gamma)  =\lVert X \rVert_\F =\sqrt{2}\lVert B \rVert_\F.
\end{equation}
\end{corollary}

The Grassmannian is geodesically complete and so any two points can be joined by a length-minimizing geodesic. In the next proposition, we will derive an explicit expression for such a geodesic in the involution model. By \eqref{eq:iso}, there will be no loss of generality in assuming that $k \le n/2$ in the following --- if $k > n/2$, then we just replace $k$ by $n-k$.
\begin{proposition}[Geodesics II]\label{prop:geo}
Let $k \le n/2$. Let $Q_0,Q_1\in \Gr(k,n)$ with eigendecompositions
$Q_0 = V_0 I_{k,n-k} V_0^\tp$ and $Q_1 = V_1 I_{k,n-k} V_1^\tp$. Let the \textsc{cs} decomposition of $V_0^\tp V_1 \in \O(n)$ be
\begin{equation}\label{eq:csd}
V_0^\tp V_1 = 
\begin{bmatrix}
U & 0 \\
0 & V
\end{bmatrix} 
\begin{bmatrix}
\cos \Theta & \sin \Theta & 0  \\
-\sin \Theta & \cos \Theta & 0 \\
0 & 0 & I_{n-2k}
\end{bmatrix}
\begin{bmatrix}
W & 0 \\
0 & Z
\end{bmatrix}^\tp
\end{equation}
where $U,W\in \O(k)$, $V,Z\in \O(n-k)$, $\Theta = \diag(\theta_1,\dots,\theta_k)\in \mathbb{R}^{k \times k}$. Then the geodesic $\gamma$ connecting $Q_0$ to $Q_1$ is 
\[
\gamma (t) = V_0 \exp \left( \frac{t}{2}\begin{bmatrix}
0 & -B \\
B^\tp & 0
\end{bmatrix} \right) I_{k,n-k} \exp \left( \frac{t}{2}\begin{bmatrix}
0 & B \\
-B^\tp & 0
\end{bmatrix} \right) V_0^\tp,
\]
where $B =-2 U [\Theta,0] V^\tp \in \mathbb{R}^{k \times (n-k)}$ with $0 \in \mathbb{R}^{k \times (n-2k)}$.
\end{proposition}
\begin{proof}
By Theorem~\ref{thm:geo}, $\gamma$ is a geodesic curve emanating from $\gamma(0) =  V_0 I_{k,n-k} V_0^\tp = Q_0$. It remains to verify that
\[
\gamma(1) =V_0 \exp \left(\frac{1}{2} \begin{bmatrix}
0 &  -B \\
B^\tp & 0
\end{bmatrix} \right) I_{k,n-k} \exp \left( \frac{1}{2}\begin{bmatrix}
0 & B \\
-B^\tp & 0
\end{bmatrix} \right) V_0^\tp = Q_1,
\]
when $B =-2 U [\Theta,0] V^\tp $. Substituting the expression for $B$,
\[
\begin{adjustbox}{width=\textwidth,totalheight=\textheight,keepaspectratio}
$
\begin{aligned}
\gamma(1) &=V_0 \begin{bmatrix}
U & 0 \\
0 & V
\end{bmatrix} 
\exp \left( \begin{bmatrix}
0 &  \Theta & 0  \\
- \Theta & 0 & 0 \\
0 & 0 & 0
\end{bmatrix} \right) I_{k,n-k} \exp \left( \begin{bmatrix}
0 & - \Theta & 0 \\
 \Theta & 0 & 0 \\
 0 & 0 & 0
\end{bmatrix} \right)
\begin{bmatrix}
U^\tp & 0 \\
0 & V^\tp
\end{bmatrix} V_0^\tp       \\
& =V_0 \begin{bmatrix}
U & 0 \\
0 & V
\end{bmatrix} 
\begin{bmatrix}
\cos \Theta & \sin \Theta & 0  \\
-\sin \Theta & \cos \Theta & 0 \\
0 & 0 & I_{n-2k}
\end{bmatrix} I_{k,n-k} \begin{bmatrix}
\cos \Theta & -\sin \Theta & 0 \\
\sin \Theta & \cos \Theta & 0 \\
 0 & 0 & I_{n-k}
\end{bmatrix}
\begin{bmatrix}
U^\tp & 0 \\
0 & V^\tp
\end{bmatrix} V_0^\tp       \\
& =V_0 \begin{bmatrix}
U & 0 \\
0 & V
\end{bmatrix} 
\begin{bmatrix}
\cos \Theta & \sin \Theta & 0  \\
-\sin \Theta & \cos \Theta & 0 \\
0 & 0 & I_{n-2k}
\end{bmatrix}
\begin{bmatrix}
W^\tp & 0 \\
0 & Z^\tp
\end{bmatrix} I_{k,n-k}
\begin{bmatrix}
W & 0 \\
0 & Z
\end{bmatrix}
\begin{bmatrix}
\cos \Theta & -\sin \Theta & 0 \\
\sin \Theta & \cos \Theta & 0 \\
 0 & 0 & I_{n-k}
\end{bmatrix}
\begin{bmatrix}
U^\tp & 0 \\
0 & V^\tp
\end{bmatrix} V_0^\tp
\end{aligned}
$
\end{adjustbox}
\]
where the last equality holds because we have
\[
I_{k,n-k} = \begin{bmatrix}
W^\tp & 0 \\
0 & Z^\tp
\end{bmatrix} I_{k,n-k}
\begin{bmatrix}
W & 0 \\
0 & Z
\end{bmatrix}
\]
whenever $W \in \O(k)$ and $Z \in \O(n-k)$. By \eqref{eq:csd}, the last expression of $\gamma(1)$ equals
\[
V_0 (V_0^\tp V_1) I_{k,n-k} (V_0^\tp V_1)^\tp V_0^\tp = V_1 I_{k,n-k} V_1^\tp = Q_1. \qedhere
\]
\end{proof}
The geodesic expression in Proposition~\ref{prop:geo} requires a \textsc{cs} decomposition \cite{GNS, Stewart2} and is more expensive to evaluate than the one in Theorem~\ref{thm:geo}. Nevertheless, we do not need Proposition~\ref{prop:geo} for our optimization algorithms in Section~\ref{sec:algo}, although its next corollary could be useful if one wants to design proximal gradient methods in the involution model.
\begin{corollary}[Geodesic distance]\label{cor:geodist}
The geodesic distance between $Q_0, Q_1 \in \Gr(k,n)$ is given by
\begin{equation}\label{eq:geodist}
d(Q_0,Q_1) =2\sqrt{2} \Bigl( \sum\nolimits_{i=1}^k  \sigma_i( B )^2 \Bigr)^{1/2} =2\sqrt{2} \Bigl( \sum\nolimits_{i=1}^k  \theta_i \Bigr)^{1/2}
\end{equation}
where $B \in \mathbb{R}^{k \times (n-k)}$ and $\Theta \in \mathbb{R}^{k \times k}$ are as in Proposition~\ref{prop:geo}.
\end{corollary}
\begin{proof}
By \eqref{eq:arc}, $L(\gamma) = \sqrt{2} \lVert B \rVert_\F = 2\sqrt{2}\lVert \Theta \rVert_\F$ with $B = -2 U [\Theta,0] V^\tp$ as in Proposition~\ref{prop:geo}.
\end{proof}
The last expression in \eqref{eq:geodist} differs from the expression in \cite[Section~4.3]{EAS} by a factor of $2\sqrt{2}$, which is exactly what we expect since the metrics in the involution and orthogonal models differ by a factor of $(2\sqrt{2})^2 = 8$, as we saw in the proof of Proposition~\ref{prop:isometric}.

The notion of a logarithmic map is somewhat less standard and we remind readers of its definition. Given a Riemannian manifold $M$ and a point $x\in M$, there exists some $r >0$ such that the exponential map $\exp_x: \B_r (0) \to M$ is a diffeomorphism on the ball $\B_r (0) \subseteq \T_x M$  of radius $r$ centered at the origin \cite[Theorem~3.7]{doCarmo}. The \emph{logarithm map}, sometimes called the \emph{inverse exponential map}, is then defined on the diffeomorphic image $\exp_x \bigl(\B_r(0)\bigr) \subseteq M$ by
\[
\log_x: \exp_x \bigl(\B_r(0)\bigr) \to \T_x M,\quad \log_x (v) \coloneqq \exp_x^{-1}(v)
\]
for all $v \in \exp_x \bigl(\B_r(0)\bigr)$.
The largest $r$ so that $\exp_x$ is a diffeomorphism on $\B_r (0)$ is the \emph{injectivity radius} at $x$ and its infimum over all $x \in M$ is the injectivity radius of $M$.
\begin{corollary}[Logarithmic map]\label{cor:log}
Let $Q_0, Q_1  \in \Gr(k,n)$ be such that $d(Q_0,Q_1) < \sqrt{2}\pi$. Let $V_0, V_1 \in \O(n)$, and $B \in \mathbb{R}^{k \times (n-k)}$ be as in Proposition~\ref{prop:geo}. The logarithmic map at $Q_0$ of $Q_1$ is
\[
\log_{Q_0} (Q_1) = V_0 \begin{bmatrix}
0 & -B \\
B^\tp & 0
\end{bmatrix} V^\tp_0.
\]
\end{corollary}
\begin{proof}
The injectivity radius of $\Gr(k,n)$ is well known to be $\pi/2$ \cite{Wong}. Write $\B_r(0) = \{ X \in \T_{Q_0} \Gr(k,n) : \lVert X \rVert_Q < r\}$ and $\B_r^d (Q_0) = \{ Q \in  \Gr(k,n) : d(Q_0,Q) < r\}$. By Corollaries~\ref{cor:exp} and \ref{cor:geodist},
\[
\exp_{Q_0} \bigl(\B_{\pi/2} (0)\bigr) = \B_{\sqrt{2} \pi}^d (Q_0).
\]
By Corollary~\ref{cor:exp} and Proposition~\ref{prop:geo},  $\log_{Q_0} : \B_{\sqrt{2} \pi} (Q_0) \to \Gr(k,n)$  has the required expression.
\end{proof}

We end this section with the expression for the parallel transport of a vector $Y$ along a geodesic $\gamma$ at a point $Q$ in the direction $X$. This will be an essential ingredient for conjugate gradient and Newton methods in the involution model (see Algorithms~\ref{alg:nt} and \ref{alg:cg}).
\begin{proposition}[Parallel transport]\label{prop:pt}
Let $Q\in \Gr(k,n)$ and $X,Y \in \T_Q \Gr(k,n)$ with
\[
Q = V I_{k,n-k} V^\tp, \qquad
X = V\begin{bmatrix}
0 & B \\
B^\tp & 0
\end{bmatrix}V^\tp,
\qquad
Y = V\begin{bmatrix}
0 & C \\
C^\tp & 0
\end{bmatrix}V^\tp,
\]
where $V \in \O(n)$ and $B, C\in \mathbb{R}^{k \times (n-k)}$. 
Let $\gamma$ be a geodesic curve emanating from $Q$ in the direction $X$. Then the parallel transport of $Y$ along $\gamma$ is 
\begin{equation}\label{eq:pt}
Y (t) =V 
\exp \left( \frac{t}{2}\begin{bmatrix}
0 & -B \\
B^\tp & 0
\end{bmatrix}\right) \begin{bmatrix}
0 & C \\
C^\tp & 0
\end{bmatrix} \exp \left(\frac{t}{2} \begin{bmatrix}
0 & B \\
-B^\tp & 0
\end{bmatrix} \right) V^\tp.
\end{equation}
\end{proposition}
\begin{proof}
Let $\gamma$ be parametrized as in \eqref{eq:geo}. A vector field $Y(t)$ that is parallel along $\gamma(t)$ may, by \eqref{eq:tanvec}, be written in the form
\[
Y(t) = V 
\exp \left( \frac{t}{2}\begin{bmatrix}
0 & -B \\
B^\tp & 0
\end{bmatrix}\right) \begin{bmatrix}
0 & C(t) \\
C(t)^\tp & 0
\end{bmatrix} \exp \left( \frac{t}{2}\begin{bmatrix}
0 & B \\
-B^\tp & 0
\end{bmatrix} \right) V^\tp
\]
for some curve $C(t)$ in $\mathbb{R}^{k \times (n-k)}$ with $C(0) = C$.
Differentiating $Y(t)$ gives
\[
\begin{adjustbox}{width=\textwidth,totalheight=\textheight,keepaspectratio}
$
\dot{Y}(t) = V 
\exp \biggl(\dfrac{t}{2} \begin{bmatrix}
0 & -B \\
B^\tp & 0
\end{bmatrix}\biggr) 
\begin{bmatrix}
- \frac{1}{2} \bigl(BC(t)^\tp + C(t)B^\tp\bigr) & \dot{C}(t)  \\
\dot{C}(t)^\tp & \frac{1}{2}\bigl(B^\tp C(t) + C(t)^\tp B\bigr)
\end{bmatrix}
\exp \biggl( \dfrac{t}{2} \begin{bmatrix}
0 & B \\
-B^\tp & 0
\end{bmatrix} \biggr) V^\tp.
$
\end{adjustbox}
\]
Since $Y(t)$ is parallel along $\gamma(t)$, we must have 
\[
\proj^{\T}_{\gamma(t)} \bigl(\dot{Y}(t)\bigr) = 0,
\]
which implies that $\dot{C}(t) = 0$ and thus $C(t) = C(0) = C$, giving us \eqref{eq:pt}. 
\end{proof}
A word about our notation for parallel transport, or rather, the lack of one. Note that $Y(t)$ depends on $\gamma$ and to indicate this dependence, we may write $Y_\gamma(t)$. Other common notations include $\tau_t Y$ \cite{Helgason1978}, $P^\gamma_t Y$ \cite{Jost}, $\gamma^t_s(Y)$ \cite{KN2} ($s=0$ for us) but there is no single standard notation.

\section{Gradient and Hessian}\label{sec:gradHess}

We now derive expressions for the Riemannian gradient and Hessian of a $C^2$ function $f : \Gr(k,n)\to \mathbb{R}$ in the involution model with \eqref{eq:tan2} for tangent vectors. As a reminder, this means:
\begin{equation}\label{eq:remind}
\begin{aligned}
\Gr(k,n) &= \{Q \in \mathbb{R}^{n \times n}  : Q^\tp Q = I,\; Q^\tp = Q, \; \tr(Q) = 2k -n\},\\
\T_Q \Gr(k,n) &= \Bigl\lbrace
V \begin{bmatrix}
0 & B \\
B^\tp & 0
\end{bmatrix}V^\tp  \in \mathbb{R}^{n \times n} : B \in \mathbb{R}^{k\times (n-k)}
\Bigr\rbrace,
\end{aligned}
\end{equation}
where $Q = V I_{k,n-k} V^\tp$.

Let $Q \in \Gr(k,n)$. Then the \emph{Riemannian gradient} $\grad f$ at $Q$ is a tangent vector $\grad f(Q) \in \T_Q \Gr(k,n)$ and, depending on context,  the \emph{Riemannian Hessian} at  $Q$ is a bilinear map:
\[
\Hess f (Q) :  \T_Q \Gr(k,n) \times \T_Q \Gr(k,n) \to \mathbb{R}.
\]

\begin{proposition}[Riemannian gradient I]\label{prop:grad}
Let $f : \Gr(k,n)\to \mathbb{R}$ be $C^1$. For any $Q \in \Gr(k,n)$, write
\begin{equation}\label{eq:fQ}
f_Q \coloneqq \biggl[\frac{\partial f}{\partial q_{ij}} (Q) \biggr]_{i,j=1}^n \in \mathbb{R}^{n \times n}.
\end{equation}
Then
\begin{equation}\label{eq:grad}
\grad f(Q) =  \frac{1}{4}\bigl[ f_Q + f_Q^\tp - Q (f_Q + f_Q^\tp) Q \bigr].
\end{equation}
\end{proposition}
\begin{proof}
The projection of $QX \in \T_Q \mathbb{R}^{n \times n}$ to $\T_Q \O(n)$ is $Q(X-X^\tp)/2$. Therefore the projection of $f_Q \in \T_Q \mathbb{R}^{n \times n}$ to $\T_Q \O(n)$ is $(f_Q-Qf_Q^\tp Q)/2$. Composing this with the projection of $\T_Q \O(n)$ to $\T_Q \Gr(k,n)$ given in \eqref{eq:proj}, we get
\[
\grad f(Q) = \proj^{\T}_Q\biggl(\frac{f_Q-Qf_Q^\tp Q}{2}\biggr) = \frac{1}{4} \bigl(f_Q+f_Q^\tp - Q f_QQ-Q f_Q^\tp Q\bigr)
\]
as required.
\end{proof}

\begin{proposition}[Riemannian Hessian I]\label{prop:Hess}
Let $f : \Gr(k,n)\to \mathbb{R}$ be $C^2$. For any $Q = V I_{k,n-k} V^\tp \in \Gr(k,n)$, let $f_Q$ be as in \eqref{eq:fQ} and
\[
f_{QQ} (X)\coloneqq \biggl[ \sum_{i,j=1}^n \Bigl(\frac{\partial^2 f}{\partial q_{ij}\partial q_{kl}} (Q)\Bigr) x_{ij} \biggr]_{k,l=1}^n, \quad f_{QQ} (X, Y)\coloneqq \sum_{i,j, k,l=1}^n \Bigl( \frac{\partial^2 f}{\partial q_{ij}\partial q_{kl}} (Q) \Bigr) x_{ij} y_{kl}.
\]
As a bilinear map, the Hessian of $f$ at $Q$ is given by 
\begin{equation}\label{eq:Hess1}
\Hess f (Q) (X,Y) = f_{QQ} (X, Y) - \frac{1}{2} \tr \bigl( f_Q^\tp Q (XY + YX) \bigr)
\end{equation}
for any $X,Y \in  \T_Q \Gr(k,n)$.
\end{proposition}
\begin{proof}
Let $\gamma$ be a geodesic curve emanating from $Q$ in the direction $X\in \T_Q \Gr(k,n)$. Then
\[
\Hess f  (Q) (X,X) =
\frac{d^2 }{dt^2}f\bigl(\gamma(t)\bigr) \biggr\rvert_{t=0} = \frac{d}{dt} \tr\bigl(f_{\gamma(t)}^\tp \dot{\gamma}(t)\bigr)\biggr\rvert_{t=0} = f_{QQ} (X) + \tr\bigl(f^\tp_Q \ddot{\gamma}(0)\bigr).
\]
Since $\gamma(t)$ is given by \eqref{eq:geo},
\[
\ddot{\gamma}(0) = V \begin{bmatrix}
-B B^\tp  & 0 \\
0 & B^\tp B
\end{bmatrix} V^\tp = -Q  \dot{\gamma}(0)^2 
\]
and so
\[
\Hess f  (Q) (X,X) = f_{QQ} (X) - \tr(f_Q^\tp Q X^2).
\]
To obtain $\Hess f(Q)$ as a bilinear map, we simply polarize the quadratic form above:
\begin{align*}
\Hess f(Q) (X,Y) &= \frac{1}{2} \bigl[ \Hess f  (Q)(X + Y,X+Y) - \Hess f  (Q)(X,X) - \Hess f (Q) (Y,Y)  \bigr] \\
&= \frac{1}{2} \biggl[ f_{QQ} (X + Y) - f_{QQ}(X) - f_{QQ}(Y) - \tr \bigl( f_Q^\tp Q (XY + YX) \bigr) \biggr] \\
&= f_{QQ} (X, Y) - \frac{1}{2} \tr \bigl( f_Q^\tp Q (XY + YX) \bigr). \qedhere
\end{align*}
\end{proof}

Our optimization algorithms require that we parameterize our tangent space as in \eqref{eq:remind} and we need to express $\grad f(Q) $ in such a form. This can be easily accomplished. Let $E_{ij}\in \mathbb{R}^{k\times(n-k)}$ be the matrix whose $(i,j)$ entry is zero and other entries are one. Let
\begin{equation}\label{eq:basis}
X_{ij} \coloneqq  V \begin{bmatrix}
0 & E_{ij} \\
E_{ij}^\tp & 0
\end{bmatrix} V^\tp  \in \T_Q \Gr(k,n).
\end{equation}
Then  $\mathcal{B}_Q \coloneqq \{ X_{ij} : i=1,\dots, k, \; j=1,\dots, n-k\}$ is an orthogonal (but not orthonormal since Riemannian norm $\lVert X_{ij} \rVert_Q = 1/\sqrt{2}$) basis of $\T_Q \Gr(k,n)$.
\begin{corollary}[Riemannian gradient II]\label{cor:grad}
Let $f$, $Q$, $f_Q$ be as in Propositions~\ref{prop:grad}. If we  partition
\begin{equation}\label{eq:partition}
V^\tp (f_Q+f_Q^\tp) V=  \begin{bmatrix}
A & B \\
B^\tp & C
\end{bmatrix},
\end{equation}
where $A \in \mathbb{R}^{k \times k}$, $B \in \mathbb{R}^{k \times (n-k)}$, $C \in \mathbb{R}^{(n-k) \times (n-k)}$, then
\begin{equation}\label{eq:gradB}
\grad f(Q) = \frac{1}{2} V \begin{bmatrix}
0 & B \\
B^\tp & 0
\end{bmatrix}V^\tp.
\end{equation}
\end{corollary}
\begin{proof}
By \eqref{eq:partition}, we may rewrite \eqref{eq:grad} as 
\[
\grad f(Q)  = \frac{1}{4} \left( 
V \begin{bmatrix}
A & B \\
B^\tp & C
\end{bmatrix}V^\tp - V \begin{bmatrix}
A & -B \\
-B^\tp & C
\end{bmatrix}V^\tp
\right) = \frac{1}{2} V \begin{bmatrix}
0 & B \\
B^\tp & 0
\end{bmatrix}V^\tp. \qedhere
\]
\end{proof}
In our optimization algorithms, \eqref{eq:partition} is how we actually compute Riemannian gradients. Note that in the basis $\mathcal{B}_Q$, the gradient of $f$ is essentially given by the matrix $B/2 \in \mathbb{R}^{k \times (n-k)}$. So in algorithms that rely only on Riemannian gradients, we just need the top right block $B$, but the other blocks $A$ and $C$ would appear implicitly in the Riemannian Hessians.

We may order the basis $\mathcal{B}_Q$ lexicographically (note that $X_{ij}$'s are indexed by two indices), then the bilinear form  $\Hess f(Q)$ has the matrix representation
\begin{equation}\label{eq:Hessmatrix}
H_Q \coloneqq \begin{bmatrix} 
\Hess f(Q)(X_{11} , X_{11}) & \Hess f(Q)(X_{11} , X_{12}) & \dots & \Hess f(Q)(X_{11} , X_{k,n-k}) \\
\Hess f(Q)(X_{12} , X_{11}) & \Hess f(Q)(X_{12} , X_{12}) & \dots & \Hess f(Q)(X_{12} , X_{k,n-k})  \\
\vdots & \vdots & \ddots & \vdots\\
\Hess f(Q)(X_{k,n-k} , X_{11}) & \Hess f(Q)(X_{k,n-k} , X_{12}) & \dots & \Hess f(Q)(X_{k,n-k} , X_{k,n-k}) 
\end{bmatrix}.
\end{equation}
In practice, the evaluation of $H_Q$ may be simplified; we will discuss this in Section~\ref{sec:Rie}.
To summarize, in the lexicographically ordered basis $\mathcal{B}_Q$,
\[
\bigl[ \grad f(Q) \bigr]_{\mathcal{B}_Q} = \frac{1}{2}\vect(B)\in  \mathbb{R}^{k(n-k)}, \qquad \bigl[ \Hess f(Q) \bigr]_{\mathcal{B}_Q} = H_Q \in  \mathbb{R}^{k(n-k)\times k(n-k)},
\]
and the Newton step $S \in \mathbb{R}^{k \times (n -k)}$ is given by the linear system
\begin{equation}\label{eq:ns}
H_Q \vect(S) = -\frac{1}{2}\vect(B). 
\end{equation}

\section{Retraction map and vector transport}\label{sec:retr}

Up till this point, everything that we have discussed is authentic Riemannian geometry, even though we have used extrinsic coordinates to obtain expressions in terms of matrices and matrix operations. This section is a departure, we will discuss two notions created for sole use in manifold optimization \cite{AMS2009}: \emph{retraction maps} and \emph{vector transports}.  They are relaxations of exponential maps and parallel transports respectively and are intended to be pragmatic substitutes in situations where these Riemannian operations are either too difficult to compute (e.g., requiring the exponential of a nonnormal matrix) or unavailable in closed form (e.g., parallel transport on a Stiefel manifold). While the involution model does not suffer from either of these problems, retraction algorithms could still serve as a good option for initializing Riemannian optimization algorithms. 

As these definitions are not found in the Riemannian geometry literature, we state a version of \cite[Definitions~4.1.1 and 8.1.1]{AMS2009} below for easy reference.
\begin{definition}[Absil--Mahony--Sepulchre]\label{def:retr}
A map $R : \T M \to M$,  $(x, v) \mapsto R_x(v)$ is a \emph{retraction} map if it satisfies the following two conditions:
\begin{enumerate}[\upshape (a)]
\item $R_x(0) = x$ for all $x \in M$;
\item\label{it:b} $d R_x(0) : \T_x M \to \T_x M$ is the identity map for all $x \in M$.
\end{enumerate}
A map $T : \T M\oplus \T M \to \T M$ associated to a retraction map $R$ is a \emph{vector transport} if it satisfies the following three conditions:
\begin{enumerate}[\upshape (i)]
\item\label{it:compat} $T(x,v,w) =\bigl(R_x(v), T_{x, v}(w)\bigr)$ for all $x \in M$ and $v,w \in \T_x M$;

\item $ T_{x, 0}(w) = w$ for all $x \in M$ and $w \in \T_x M$;

\item $T_{x, v}(a_1w_1+a_2w_2) = a_1 T_{x, v}(w_1) + a_2 T_{x, v}(w_2)$ for all $a_1,a_2 \in \mathbb{R}$, $x \in M$, and $v,w_1,w_2 \in \T_{x} M$.
\end{enumerate}
\end{definition}
The condition \eqref{it:compat} says that the vector transport $T$ is compatible with its retraction map $R$, and also defines the map $T_{x,v} : \T_x M \to \T_x M$. Note that $v$ is the direction to move in  while $w$ is the vector to be transported.

For the purpose of optimization, we just need $R$ and $T$ to be well-defined on a neighbourhood of $M \cong \{(x, 0) \in \T M\} \subseteq \T M$ and $M \cong \{(x, 0, 0) \in \T M \oplus \T M  \}\subseteq \T M\oplus \T M$ respectively. If $R$ and $T$ are $C^1$ maps, then various optimization algorithms relying on $R$ and $T$  can be shown to converge \cite{AMS2009}, possibly under the additional assumption that $M$ has nonnegative \cite{DDO} or bounded sectional curvature \cite{SFF}. In particular, these results apply in our case since being a compact symmetric space, $\Gr(k ,n)$ has both nonnegative and bounded sectional curvature \cite{CE, Ziller}.

\begin{example}[Projection as retraction]
For a manifold $M$ embedded in Euclidean space $\mathbb{R}^n$ or $\mathbb{R}^{m \times n}$,  we may regard tangent vectors in $\T_x M$ to be of the form $x + v$. In this case an example of a retraction map is given by the projection of tangent vectors onto $M$,
\[
R_x (v) = \argmin_{y\in M}\, \lVert x+v-y\rVert,
\]
where $\lVert \, \cdot \, \rVert$ is either the $2$- or Frobenius norm.
By \cite[Lemma~3.1]{AM2012}, the map $R_x$ is well-defined for small $v$ and is a retraction.
\end{example}

We will give three retraction maps for $\Gr(k,n)$ that are readily computable in the involution model with \textsc{evd}, block \textsc{qr}, and Cayley transform respectively.  The latter two are inspired by similar maps defined for the projection model in \cite{HHT} although our motivations are somewhat different.

We begin by showing how one may compute the projection $\argmin \bigl\{ \lVert A-Q\rVert_\F : Q\in \Gr(k, n) \bigr\}$ for an arbitrary matrix $A \in \mathbb{R}^{n \times n}$ in the involution model, a result that may be of independent interest.
\begin{lemma}\label{lem:approx}
Let $A\in \mathbb{R}^{n\times n}$ and
\begin{equation}\label{eq:approx}
\frac{A+A^\tp}{2} = VDV^\tp
\end{equation}
be an eigendecomposition with $V \in \O(n)$ and $D = \diag(\lambda_1,\dots,\lambda_n)$, $\lambda_1 \ge  \dots \ge \lambda_n$. Then $Q = VI_{k, n-k}V^\tp$ is a minimizer of
\[
\min \bigl\{ \lVert A-Q\rVert_\F : Q^\tp Q = I, \; Q^\tp = Q, \; \tr(Q) = 2k -n \bigr\}.
\]
\end{lemma}
\begin{proof}
Since $Q$ is symmetric, $\lVert A-Q\rVert_\F^2 = \lVert (A+A^\tp)/2-Q\rVert_\F^2 + \lVert (A-A^\tp)/2\rVert_\F^2$, a best approximation to $A$ is also a best approximation to $(A+A^\tp)/2$. By \eqref{eq:approx}, $\lVert (A+A^\tp)/2-Q\rVert_\F = \lVert D-V^\tp Q V\rVert_\F$ and so for a best approximation $V^\tp Q V$ must be a diagonal matrix. Since the eigenvalues $\delta_1,\dots,\delta_n$ of a symmetric orthogonal $Q$ must be $\pm 1$ and $\tr(Q) = 2k - n$, the multiplicities of  $+1$ and $-1$  are $k$ and $n-k$ respectively. By assumption, $\lambda_1 \ge \dots \ge \lambda_n$, so
\[
\min_{\delta_1 + \dots + \delta_n =2k -n} (\lambda_1 -\delta_1)^2 + \dots + (\lambda_n -\delta_n)^2
\]
is attained when $\delta_1 = \dots = \delta_k = +1$ and $\delta_{k+1} = \dots = \delta_n = -1$. Hence $V^\tp Q V = \diag(\delta_1,\dots,\delta_n) = I_{k, n-k}$ as required.
\end{proof}
It is clear from the proof, which is a variation of standard arguments \cite[Section~8.1]{Higham2}, that a minimizer is not unique if and only if $\lambda_k=\lambda_{k+1}$, i.e., the $k$th and $(k+1)$th eigenvalues of $(A+A^\tp)/2$ coincide. Since any $Q \in \Gr(k, n)$ by definition has $\lambda_k = +1 \ne -1 =\lambda_{k+1}$, the projection is always unique in a small enough neighborhood of $Q$ in $\mathbb{R}^{n \times n}$.

In the following, let $\eig: \mathbb{R}^{n \times n} \to \O(n)$ be the map that takes any $A  \in \mathbb{R}^{n \times n}$ to an orthogonal matrix of eigenvectors of $(A+A^\tp)/2$.
\begin{proposition}[Retraction I]\label{prop:Reig}
Let $Q\in \Gr(k,n)$ and $X,Y \in \T_Q \Gr(k,n)$ with
\begin{equation}\label{eq:form}
Q = V I_{k,n-k} V^\tp, \qquad
X = V\begin{bmatrix}
0 & B \\
B^\tp & 0
\end{bmatrix}V^\tp,
\qquad
Y = V\begin{bmatrix}
0 & C \\
C^\tp & 0
\end{bmatrix}V^\tp,
\end{equation}
where $V \in \O(n)$ and $B, C\in \mathbb{R}^{k \times (n-k)}$. 
Then
\begin{align*}
R_Q^\eig(X) &= V\eig\bigg(\begin{bmatrix}
I & B \\
B^\tp & -I
\end{bmatrix}\bigg)I_{k, n-k}\eig\bigg(\begin{bmatrix}
I & B \\
B^\tp & -I
\end{bmatrix}\bigg)^\tp V^\tp
\intertext{defines a retraction and}
T_Q^\eig (X,Y) &= V\eig\bigg(\begin{bmatrix}
I & B \\
B^\tp & -I
\end{bmatrix}\bigg)\begin{bmatrix}
0 & C \\
C^\tp & 0
\end{bmatrix} \eig\bigg(\begin{bmatrix}
I & B \\
B^\tp & -I
\end{bmatrix}\bigg)^\tp V^\tp
\end{align*}
defines a vector transport.
\end{proposition}
\begin{proof}
It follows from Lemma~\ref{lem:approx} that $R_Q^\eig$ defines a projection. The properties in Definition~\ref{def:retr} are routine to verify.
\end{proof}
As we will see later, the exponential map in our Riemannian algorithms may be computed in $O\bigl(nk(n-k)\bigr)$ time, so a retraction map that requires an \textsc{evd} offers no advantage. Furthermore, the eigenvector map $\eig$ is generally discontinuous \cite{Kato}, which can present a problem.
One alternative would be to approximate the map $\eig$ with a \textsc{qr} decomposition --- one should think of this as the first step of Francis's \textsc{qr} algorithm for \textsc{evd}. In fact, we will not even require a full \textsc{qr} decomposition, a $2 \times 2$ block \textsc{qr} decomposition suffices. Let $\qr : \mathbb{R}^{n \times n} \to \O(n)$ be a map that takes a matrix $A$ to its orthogonal factor in a $2 \times 2$ block \textsc{qr} decomposition, i.e.,
\[
A = \qr(A) 
\begin{bmatrix}
R_1 & R_2 \\
0 & R_3
\end{bmatrix}, \quad R_1 \in \mathbb{R}^{k \times k},\; R_2 \in \mathbb{R}^{k \times (n-k)}, \; R_3 \in \mathbb{R}^{(n-k) \times (n-k)}.
\]
Note that $\qr(A)$ is an orthogonal matrix but the second factor just needs to be block upper triangular, i.e., $R_1$ and $R_3$ are not required to be upper triangular matrices.  We could compute $\qr(A) $ with, say, the first $k$ steps of Householder \textsc{qr} applied to $A$.
\begin{proposition}[Retraction II]\label{prop:Rqr}
Let $Q\in \Gr(k,n)$ and $X,Y \in \T_Q \Gr(k,n)$ be as in \eqref{eq:form}. If $\qr$ is well-defined and differentiable near $I_{k,n-k}$ and $\qr(I_{k,n-k}) = I$, then
\begin{align*}
R_Q^\qr (X) &= V\qr\bigg(\frac{1}{2}\begin{bmatrix}
I & B \\
B^\tp & -I
\end{bmatrix}\bigg) I_{k,n-k} \qr\bigg(\frac{1}{2}\begin{bmatrix}
I & B \\
B^\tp & -I
\end{bmatrix}\bigg)^\tp V^\tp
\intertext{defines a retraction and}
T_Q^\qr (X,Y) &= V\qr\bigg(\frac{1}{2}\begin{bmatrix}
I & B \\
B^\tp & -I
\end{bmatrix}\bigg) \begin{bmatrix}
0 & C \\
C^\tp & 0
\end{bmatrix} \qr\biggl(\frac{1}{2}\begin{bmatrix}
I & B \\
B^\tp & -I
\end{bmatrix}\biggr)^\tp V^\tp
\end{align*}
defines a vector transport.
\end{proposition}
\begin{proof}
Only property \eqref{it:b} in Definition~\ref{def:retr} is not immediate and requires checking.
Let the following be a block \textsc{qr} decomposition:
\begin{equation}\label{eq:blockqr}
\frac{1}{2}
\begin{bmatrix}
I & tB \\
tB^\tp & -I
\end{bmatrix} = \begin{bmatrix}
Q_1(t) & Q_2(t) \\
Q_3(t) & Q_4(t)
\end{bmatrix}
\begin{bmatrix}
R_1(t) & R_2(t) \\
0 & R_3(t)
\end{bmatrix}
= Q(t)R(t),
\end{equation}
with $Q(t) \in \O(n)$.
Since $Q(t)Q(t)^\tp = 1$ and $Q(0)=I$, $Q'(0)$ is skew-symmetric and
\[
\frac{d}{dt} Q(t)I_{k, n-k}Q(t)^\tp\biggr\rvert_{t=0} = \begin{bmatrix}
Q_1'(0)+Q_1'(0)^\tp & -Q_2'(0) + Q_3'(0)^\tp \\
Q_3'(0)-Q_2'(0)^\tp & -Q_4'(0)-Q_4'(0)^\tp
\end{bmatrix} =  \begin{bmatrix}
0 & 2Q_3'(0)^\tp \\
2Q_3'(0) & 0
\end{bmatrix}.
\]
Comparing the $(1, 1)$ and $(2, 1)$ entries in \eqref{eq:blockqr}, we get
\[
Q_1(t)R_1(t)=I, \qquad Q_3(t)R_1(t)=tB^\tp/2.
\]
Hence $Q_3(t) = tB^\tp Q_1(t)/2$, $Q_3'(0) = B^\tp Q_1(0)/2 = B^\tp /2$, and we get
\[
\frac{d}{dt} Q(t)I_{k, n-k}Q(t)^\tp\biggr\rvert_{t=0}=\begin{bmatrix}
0 & B \\
B^\tp & 0
\end{bmatrix},
\]
as required.
\end{proof}

If we use a first-order Pad\'e approximation $\exp(X) \approx (I + X)(I - X)^{-1}$ for the matrix exponential terms in the exponential map \eqref{eq:exp} and parallel transport \eqref{eq:pt}, we obtain another retraction map and vector transport. This Pad\'e approximation is the well-known \emph{Cayley transform} $\cay$, which takes a skew-symmetric matrix to an orthogonal matrix and vice versa:
\[
\cay : \mathsf{\Lambda}^2(\mathbb{R}^n) \to \O(n), \quad \Lambda \to (I+ \Lambda)(I-\Lambda)^{-1}.
\]
\begin{proposition}[Retraction III]\label{prop:Rcay}
Let $Q\in \Gr(k,n)$ and $X,Y \in \T_Q \Gr(k,n)$ be as in \eqref{eq:form}. Then
\begin{align*}
R_Q^\cay (X) &=
V\cay \bigg(\frac{1}{4}\begin{bmatrix}
0 & -B \\
B^\tp & 0
\end{bmatrix}\bigg) I_{k,n-k}\cay \bigg(\frac{1}{4}\begin{bmatrix}
0 & -B \\
B^\tp & 0
\end{bmatrix}\bigg)^\tp V^\tp
\intertext{defines a retraction and}
T_Q^\cay(X,Y) &= V\cay\bigg(\frac{1}{4}\begin{bmatrix}
0 & -B \\
B^\tp & 0
\end{bmatrix}\bigg) \begin{bmatrix}
0 & C \\
C^\tp & 0
\end{bmatrix} \cay\bigg(\frac{1}{4}\begin{bmatrix}
0 & -B \\
B^\tp & 0
\end{bmatrix}\bigg)^\tp V^\tp
\end{align*}
defines a vector transport.
\end{proposition}
\begin{proof}
Again, only property \eqref{it:b} in Definition~\ref{def:retr} is not immediate and requires checking. But this is routine we omit the details.
\end{proof}

\section{Algorithms}\label{sec:algo}

We will now discuss optimization algorithms for minimizing a function $f : \Gr(k,n) \to \mathbb{R}$ in  the involution model.  In principle, this is equivalent to a quadratically constrained optimization problem in $n^2$ variables $[q_{ij}]_{i,j=1}^n = Q\in \mathbb{R}^{n \times n}$:
\begin{equation}\label{eq:nlp}
\begin{tabular}{rl}
minimize & $f(Q)$\\
subject to & $Q^\tp Q = I$, $Q^\tp = Q$, $\tr(Q) = 2k -n$.
\end{tabular}
\end{equation}
Nevertheless, if one attempts to minimize any of the objective functions $f$ in Section~\ref{sec:num} by treating \eqref{eq:nlp} as a general nonlinear constrained optimization problem using, say, the \textsc{Matlab} Optimization Toolbox,  every available method --- interior point, trust region, sequential quadratic programming, active set ---  will fail without even finding a feasible point, never mind a minimizer. The Riemannian geometric objects and operations of the last few sections are essential to solving \eqref{eq:nlp}.

We will distinguish between two types of optimization algorithms. The \emph{retraction algorithms}, as its name implies, will be based on various retractions and vector transports discussed in Section~\ref{sec:retr}. The \emph{Riemannian algorithms}, on the other hand, are built upon true Riemannian geodesics and parallel transports discussed in Section~\ref{sec:exp}. Both types of algorithms will rely on the materials on points in Section~\ref{sec:points}, tangent vectors and metric in Section~\ref{sec:tangent}, and Riemannian gradients and Hessians in Section~\ref{sec:gradHess}.

For both types of algorithms, the involution model offers one significant advantage over other existing models. By \eqref{eq:gradB} and \eqref{eq:exp}, at a point $Q \in \Gr(k,n)$ and in a direction $X \in \T_Q \Gr(k,n)$, the Riemannian gradient and the exponential map are
\[
\begin{adjustbox}{width=\textwidth,totalheight=\textheight,keepaspectratio}
$
\grad f(Q) = V \begin{bmatrix}
0 & G/2 \\
G^\tp/2 & 0
\end{bmatrix}V^\tp, \quad \exp_Q(X)  = V \exp \left( \begin{bmatrix}
0 & -S/2 \\
S^\tp/2 & 0
\end{bmatrix} \right) I_{k,n-k} \exp \left(\begin{bmatrix}
0 & S/2 \\
-S^\tp/2 & 0
\end{bmatrix} \right) V^\tp
$
\end{adjustbox}
\]
respectively. In the involution model, explicit parallel transport and exponential map can be avoided. Instead of $\grad f(Q)$ and $\exp_Q(X)$, it suffices to work with the matrices $G,S \in \mathbb{R}^{k \times (n-k)}$ that we will call \emph{effective gradient} and \emph{effective step} respectively, and doing so leads to extraordinarily simple and straightforward expressions in our algorithms. We will highlight this simplicity at appropriate junctures in Sections~\ref{sec:retralg} and \ref{sec:Rie}. Aside from simplicity, a more important consequence is that all key computations in our algorithms are performed at the \emph{intrinsic dimension} of $\Gr(k,n)$. Our steepest descent direction, conjugate direction, Barzilai--Borwein step, Newton step, quasi-Newton step, etc, would all be represented as $k(n-k)$-dimensional objects. This is a feature not found in the algorithms of \cite{AMS, EAS, HHT}.

\subsection{Initialization, eigendecomposition, and exponentiation}\label{sec:subroutines}

We begin by addressing three issues that we will frequently encounter in our optimization algorithms.

First observe that it is trivial to generate a point $Q \in \Gr(k,n)$ in the involution model: Take any orthogonal matrix $V \in \O(n)$, generated by say a \textsc{qr} decomposition of a random $n\times n$ matrix. Then we always have $Q \coloneqq V I_{k,n-k} V^\tp \in \Gr(k,n)$. We may easily generate as many random feasible initial points for our algorithms as we desire or simply take $I_{k,n-k}$ as our initial point.

The inverse operation of obtaining a $V \in \O(n)$ from a given $Q \in \Gr(k,n)$ so that
$Q = V I_{k,n-k} V^\tp$ seems more expensive as it appears to require an \textsc{evd}. In fact, by the following observation, the cost is the same --- a single \textsc{qr} decomposition.
\begin{lemma}\label{lem:qr}
Let $Q \in \mathbb{R}^{n \times n}$ with $Q^\tp Q = I$, $Q^\tp = Q$, $\tr(Q) = 2k -n$. If
\begin{equation}\label{eq:qrV}
\frac{1}{2}(I + Q) = V\begin{bmatrix} R_1 & R_2 \\ 0 & 0 \end{bmatrix}, \qquad V \in \O(n), \; R_1 \in \mathbb{R}^{k \times k},\; R_2\in \mathbb{R}^{k \times (n-k)},
\end{equation}
is a \textsc{qr} decomposition, then $Q = V I_{k,n-k} V^\tp$.
\end{lemma}
\begin{proof}
Recall from \eqref{eq:1eig} that for such a $Q$, we may write $V = [Y, Z]$ where $Y \in \V(k,n)$ and $Z \in \V(n-k,n)$ are a $+1$-eigenbasis and a $-1$-eigenbasis of $Q$ respectively. By Proposition~\ref{prop:G2}, $\frac{1}{2}(I + Q)$ is the projection matrix onto the $+1$-eigenspace $\im (Y)= \im\bigl(\frac{1}{2}(I + Q)\bigr)$, i.e., $Y$ is an orthonormal column basis for $\frac{1}{2}(I + Q)$ and is therefore given by its condensed \textsc{qr} decomposition. As for $Z$, note that any orthonormal basis for $\im(Y)^\perp$ would serve the role, i.e., $Z$ can be obtained from the full \textsc{qr} decomposition. In summary,
\[
\frac{1}{2}(I + Q) = Y \begin{bmatrix} R_1 \\ 0 \end{bmatrix} = \begin{bmatrix} Y & Z \end{bmatrix} \begin{bmatrix} R_1 & R_2\\ 0 & 0 \end{bmatrix}.
\]
As a sanity check, note that
\[
\frac{1}{2}(I + Q) = Y Y^\tp =  \begin{bmatrix} Y & Z \end{bmatrix} \begin{bmatrix} I_k & 0 \\ 0 & 0 \end{bmatrix} \begin{bmatrix} Y \\ Z \end{bmatrix} = V \begin{bmatrix} I_k & 0 \\ 0 & 0 \end{bmatrix} V^\tp,
\]
and therefore
\[
Q = V \begin{bmatrix} I_k & 0 \\ 0 & -I_{n-k} \end{bmatrix} V^\tp = V I_{k,n-k} V^\tp. \qedhere
\]
\end{proof}
Our expressions for tangent vector, exponential map, geodesic, parallel transport, retraction, etc, at a point  $Q \in \Gr(k,n)$  all involve its matrix of eigenvectors $V \in \O(n)$. So Lemma~\ref{lem:qr} plays an important role in our algorithms.
In practice, numerical stability considerations in the presence of rounding errors \cite[Section~3.5.2]{Demmel} require that we perform our \textsc{qr} decomposition with \emph{column pivoting} so that \eqref{eq:qrV} becomes
\[
\frac{1}{2}(I + Q) = V\begin{bmatrix} R_1 & R_2 \\ 0 & 0 \end{bmatrix} \Pi^\tp
\]
where $\Pi$ is a permutation matrix. This does not affect our proof above; in particular, note that we have no need for $R_1$ nor $R_2$ nor $\Pi$ in any of our algorithms.

The most expensive step in our Riemannian algorithms is the evaluation
\begin{equation}\label{eq:expB}
B \mapsto
\exp \biggl( \begin{bmatrix}
0 & B \\
-B^\tp & 0
\end{bmatrix} \biggr)
\end{equation}
for $B \in \mathbb{R}^{k \times (n-k)}$. General algorithms for computing matrix exponential \cite{Higham2, MVL} do not exploit structures aside from normality. There are specialized algorithms that take advantage of skew-symmetry\footnote{The retraction based on Cayley transform in Proposition~\ref{prop:Rcay} may be viewed as a special case of the Pad\'e approximation method in \cite{CL}.} \cite{CL} or both skew-symmetry and sparsity \cite{DLP} or the fact \eqref{eq:expB} may be regarded as the exponential map of a Lie algebra to a Lie group \cite{CI}, but all of them require $O(n^3)$ cost. In \cite{EAS}, the exponential is computed via an \textsc{svd} of $B$.

Fortunately for us, we have a fast algorithm for \eqref{eq:expB} based on Strang splitting \cite{Strang} that takes time at most $12nk(n-k)$. First observe a matrix in the exponent of \eqref{eq:expB} may be written as a unique linear combination
\begin{equation}\label{eq:split}
\begin{bmatrix} 0 & B \\ \smash[b]{-B^\tp} & 0 \end{bmatrix} = \sum_{i=1}^k \sum_{j=1}^{n-k} \alpha_{ij} \begin{bmatrix}
0 & E_{ij} \\
-E_{ij}^\tp & 0
\end{bmatrix}
\end{equation}
where $\alpha_{ij} \in \mathbb{R}$ and $E_{ij}$ is the matrix whose $(i,j)$ entry is one and other entries are zero.  Observe that
\[
\exp \biggl(\theta \begin{bmatrix}
0 & E_{ij} \\
-E_{ij}^\tp & 0
\end{bmatrix} \biggr)
 =
\begin{bmatrix}
I + (\cos \theta -1) E_{ii} & (\sin \theta) E_{ij} \\
-(\sin \theta) E_{ji} & I+(\cos \theta-1) E_{jj} 
\end{bmatrix} \eqqcolon G_{i,j+k}(\theta)
\]
is a Givens rotation in the $i$th and $(j+k)$th plane of $\theta$ radians \cite[p.~240]{GVL}.
Strang splitting, applied recursively to \eqref{eq:split}, then allows us to approximate
\begin{multline}\label{eq:strang}
\exp \biggl( \begin{bmatrix}
0 & B \\
-B^\tp & 0
\end{bmatrix} \biggr)
\approx 
 G_{1,1+k}\bigl(\tfrac{1}{2} \alpha_{11}\bigr) G_{1,2+k}\bigl(\tfrac{1}{2} \alpha_{12}\bigr) \cdots  G_{k,n-1}\bigl(\tfrac{1}{2} \alpha_{k,n-k-1}\bigr) G_{k,n}\bigl(\alpha_{k,n-k}\bigr)\\
G_{k,n-1}\bigl(\tfrac{1}{2} \alpha_{k,n-k-1}\bigr) \cdots G_{1,2+k}\bigl(\tfrac{1}{2} \alpha_{12}\bigr)  G_{1,1+k}\bigl(\tfrac{1}{2} \alpha_{11}\bigr).
\end{multline}
Computing the product in \eqref{eq:strang} is thus equivalent to computing a sequence of $2k(n-k)-1$ Givens rotations, which takes time $12nk(n-k) -6n$. For comparison, directly evaluating \eqref{eq:expB} via an \textsc{svd} of $B$ would have taken  time $4k(n-k)^2+22k^3 + 2n^3$ (first two summands for \textsc{svd}  \cite[p.~493]{GVL}, last summand for two matrix-matrix products).

The approximation in \eqref{eq:strang} requires that $\lVert B \rVert $ be sufficiently small \cite{Strang}. But as gradient goes to zero when the iterates converge to a minimizer, $\lVert B \rVert $ will eventually be small enough for Strang approximation. We initialize our Riemannian algorithms with retraction algorithms, which do not require matrix exponential, i.e., run a few steps of a retraction algorithm to get close to a minimizer before switching to a Riemannian algorithm.

\subsection{Retraction algorithms}\label{sec:retralg}

In manifold optimization algorithms, an iterate is a point on a manifold and a search direction is a tangent vector at that point. Retraction algorithms rely on the retraction map $R_Q$ for updating iterates and vector transport $T_Q$ for updating search directions.  Our interest in retraction algorithms is primarily to use them to initialize the Riemannian algorithms in the next section, and as such we limit ourselves to the least expensive ones.

A retraction-based steepest descent avoids even vector transport and takes the simple form
\[
Q_{i+1} = R_{Q_i}\bigl(-\alpha_i \grad f(Q_i)\bigr),
\]
an analogue of the usual $x_{i+1} = x_i - \alpha_i \grad f(x_i)$ in Euclidean space. As for our choice of retraction map, again computational costs dictate that we exclude the projection $R_Q^\eig$ in Proposition~\ref{prop:Reig} since it requires an \textsc{evd}, and limit ourselves to the \textsc{qr} retraction $R_Q^\qr$ or Cayley retraction $R_Q^\cay$ in Propositions~\ref{prop:Rqr} and \ref{prop:Rcay} respectively. We present the latter in Algorithm~\ref{alg:cay} as an example.

We select our step size  $\alpha_i$ using the well-known Barzilai--Borwein formula \cite{BB} but any line search procedure may be used instead.  Recall that over Euclidean space, there are two choices for the Barzilai--Borwein step size:
\begin{equation}\label{eq:bb}
\alpha_i = \frac{s_{i-1}^\tp s_{i-1}}{(g_i-g_{i-1})^\tp s_{i-1}},\qquad
\alpha_i = \frac{(g_i-g_{i-1})^\tp s_{i-1}}{(g_i-g_{i-1})^\tp (g_i-g_{i-1})},
\end{equation}
where $s_{i-1} \coloneqq x_i - x_{i-1}$.
On a manifold $M$, the gradient $g_{i-1} \in \T_{x_{i-1}} M$ would have to be first parallel transported to $\T_{x_i} M$ and the step $s_{i-1}$  would need to be replaced by a tangent vector in $\T_{x_{i-1}} M$ so that the exponential map $\exp_{x_{i-1}}(s_{i-1}) = x_i$. Upon applying this procedure, we obtain
\begin{equation}\label{eq:bbM}
\alpha_i = \frac{\tr(S_{i-1}^\tp S_{i-1})}{\tr\bigl( (G_i-G_{i-1})^\tp S_i)\bigr)}, \qquad
\alpha_i = \frac{\tr\bigl((G_i-G_{i-1})^\tp S_{i-1}\bigr)}{\tr\bigl((G_i-G_{i-1})^\tp (G_i-G_{i-1})\bigr)}.
\end{equation}
In other words, it is as if we have naively replaced the $g_i$ and $s_i$ in \eqref{eq:bb} by the effective gradient $G_i$ and the effective step $S_i$. But \eqref{eq:bbM} is indeed the correct Riemannian expressions for Barzilai--Borwein step size in the involution model --- the parallel transport and exponential map have already been taken into account when we derive \eqref{eq:bbM}. This is an example of the extraordinary simplicity of the involution model that we mentioned earlier and will see again in Section~\ref{sec:Rie}.

\begin{algorithm}
\caption{Steepest descent with Cayley retraction}\label{alg:cay}
\begin{algorithmic}[1]
\State Initialize $Q_0 = V_0 I_{k,n-k} V_0^\tp\in \Gr(k,n)$.
\For{$i=0,1,\dots $}
\State compute effective gradient $G_i$ at $Q_i$  \Comment{entries $\ast$ not needed}
\[
V_i^\tp (f_{Q_i}+f_{Q_i}^\tp) V_i=  \begin{bmatrix}
* & 2G_i \\
2G_i^\tp & *
\end{bmatrix};
\]

\If {$i=0$}
	\State initialize $S_0 = -G_0$, $\alpha_0 = 1$;
\Else 
\State compute Barzilai--Borwein step \Comment{or get $\alpha_i$ from line search}
\begin{align*}
\alpha_i &= \tr\bigl((G_i-G_{i-1})^\tp S_{i-1}\bigr)/\tr\bigl((G_i-G_{i-1})^\tp (G_i-G_{i-1})\bigr);\\
S_i &= -\alpha_i G_i;
\end{align*}
\EndIf

\State perform Cayley transform
\[
C_i = 
\begin{bmatrix}
I &  S_i/4 \\
- S_i^\tp/4 & I
\end{bmatrix}
\begin{bmatrix}
I & -S_i/4\\
S_i^\tp/4 & I
\end{bmatrix}^{-1};
\]

\State update eigenbasis \Comment{effective vector transport}
\[
V_{i+1} = V_i C_i;
\]

\State update iterate
\[
Q_{i+1} = V_{i+1}  I_{k,n-k}  V_{i+1}^\tp;
\]
\EndFor
\end{algorithmic}
\end{algorithm}

Of the two expressions for $\alpha_i$ in \eqref{eq:bbM}, we chose the one on the right because our effective gradient $G_i$, which is computed directly, is expected to be slightly more accurate than our effective step size $S_i$, which is computed from $G_i$.
Other more sophisticated retraction algorithms \cite{AMS2009} can be readily created for the involution model using the explicit expressions derived in Section~\ref{sec:retr}.

\subsection{Riemannian algorithms}\label{sec:Rie}

Riemannian algorithms, called ``geometric algorithms'' in \cite{EAS}, are true geometric analogues of those on Euclidean spaces --- straight lines are replaced by geodesic curves, displacements by parallel transports,  inner products by Riemannian metrics, gradients and Hessians by their Riemannian counterparts. Every operation in a Riemannian algorithm is intrinsic: iterates stay on the manifold, conjugate and search directions stay in tangent spaces, and there are no geometrically meaningless operations like adding a point to a tangent vector or subtracting tangent vectors from two different tangent spaces.

The involution model, like other models in \cite{AMS,EAS,HHT}, supplies a system of extrinsic coordinates that allow geometric objects and operations to be computed with standard numerical linear algebra but it offers a big advantage, namely, one can work entirely with the effective gradients and effective steps. For example, it looks as if parallel transport is missing from our Algorithms~\ref{alg:sd}--\ref{alg:qn}, but that is only because the expressions in the involution model can be simplified to an extent that gives such an illusion. Our parallel transport is effectively contained in the step where we update the eigenbasis $V_i$ to $V_{i+1}$.

We begin with steepest descent in Algorithm~\ref{alg:sd}, the simplest of our four Riemannian algorithms. As in the case of Algorithm~\ref{alg:cay}, we will use Barzilai--Borwein step size but any line search procedure may be used to produce $\alpha_i$. In this case, any conceivable line search procedure would have required us to search over a geodesic curve and thus having to evaluate matrix exponential multiple times, using the Barzilai--Borwein step size circumvents this problem entirely.

Unlike its retraction-based counterpart in Algorithm~\ref{alg:cay}, here the iterates descent along geodesic curves. Algorithm~\ref{alg:cay} may in fact be viewed as an approximation of Algorithm~\ref{alg:sd} where the matrix exponential in Step~9 is replaced with its first-order Pad\'e approximation, i.e., a Cayley transform.
\begin{algorithm}
\caption{Steepest descent}\label{alg:sd}
\begin{algorithmic}[1]
\State Initialize $Q_0 =V_0 I_{k,n-k} V_0^\tp\in \Gr(k,n)$.
\For{$i=0,1,\dots $}
\State compute effective gradient $G_i$ at $Q_i$ \Comment{entries $\ast$ not needed}
\[
V_i^\tp (f_{Q_i}+f_{Q_i}^\tp) V_i=  \begin{bmatrix}
* & 2G_i \\
2G_i^\tp & *
\end{bmatrix};
\]

\If {$i=0$}
	\State initialize $S_0 = -G_0$, $\alpha_0 = 1$;
\Else 
\State compute Barzilai--Borwein step \Comment{or get $\alpha_i$ from line search}
\begin{align*}
\alpha_i &= \tr\bigl((G_i-G_{i-1})^\tp S_{i-1}\bigr)/\tr\bigl((G_i-G_{i-1})^\tp (G_i-G_{i-1})\bigr);\\
S_i &= -\alpha_i G_i;
\end{align*}
\EndIf
\State update eigenbasis \Comment{effective parallel transport}
\[
V_{i+1} = V_i \exp \left( \begin{bmatrix}
0 & -S_i/2 \\
S^\tp_i/2 & 0
\end{bmatrix} \right);
\]

\State update iterate
\[
Q_{i+1} = V_{i+1}  I_{k,n-k}  V_{i+1}^\tp;
\]
\EndFor
\end{algorithmic}
\end{algorithm}

Newton method, shown in Algorithm~\ref{alg:nt}, is straightforward with the computation of Newton step as in \eqref{eq:ns}. In practice, instead of a direct evaluation of $H_Q \in \mathbb{R}^{k(n-k)\times k(n-k)}$ as in \eqref{eq:Hessmatrix}, we determine $H_Q$ in a manner similar to Corollary~\ref{cor:grad}. When regarded as a linear map $H_Q : \T_Q \Gr(k,n) \to  \T_Q \Gr(k,n)$, its value on a basis vector $X_{ij}$ in \eqref{eq:basis} is
\begin{equation}\label{eq:HQ}
H_Q(X_{ij}) = \frac{1}{4}V \begin{bmatrix}
0 & B_{ij}+AE_{ij}-E_{ij}C \\
(B_{ij}+AE_{ij}-E_{ij}C)^\tp & 0
\end{bmatrix} V^\tp,
\end{equation}
where $A,C$ are as in \eqref{eq:partition} and $B_{ij}$ is given by
\[
V^\tp \bigl(f_{QQ}(X_{ij})+f_{QQ}(X_{ij})^\tp\bigr) V=  \begin{bmatrix}
* & B_{ij} \\
B_{ij}^\tp & *
\end{bmatrix},
\]
for all  $i=1,\dots, k$, $j=1,\dots, n-k$. Note that these computations can be performed completely in parallel --- with $k(n-k)$ cores, entries of $H_Q$ can be evaluated all at once.
\begin{algorithm}
\caption{Newton's method}\label{alg:nt}
\begin{algorithmic}[1]
\State Initialize $Q_0=V_0 I_{k,n-k} V_0^\tp \in \Gr(k,n)$.
\For{$i=0,1,\dots $}
\State compute effiective gradient $G_i$ at $Q_i$
\[
V_i^\tp (f_{Q_i}+f_{Q_i}^\tp) V_i=  \begin{bmatrix}
A_i & 2G_i \\
2G_i^\tp & C_i
\end{bmatrix};
\]
\State generate Hessian matrix $H_Q$ by \eqref{eq:Hessmatrix} or \eqref{eq:HQ};

\State solve for effective Newton step $S_i$
\[
H_Q \vect(S_i) = -\vect(G_i);
\]

\State update eigenbasis \Comment{effective parallel transport}
\[
V_{i+1} = V_i \exp \left( \begin{bmatrix}
0 & S_i/2 \\
-S^\tp_i/2 & 0
\end{bmatrix} \right);
\]
\State update iterate
\[
Q_{i+1} = V_{i+1} I_{k,n-k} V_{i+1}^\tp;
\]
\EndFor
\end{algorithmic}
\end{algorithm}

Our conjugate gradient uses the Polak--Ribi\`ere formula \cite{PR} for conjugate step size; it is straightforward to replace that with the formulas of Dai--Yuan \cite{DY}, Fletcher--Reeves \cite{FR}, or Hestenes--Stiefel \cite{HS}. For easy reference:
\begin{equation}\label{eq:cg}
\begin{aligned}
\beta_i^{\textsc{pr}} &= \tr\bigl(G_{i+1}^\tp(G_{i+1}- G_i)\bigr)/\tr(G_i^\tp G_i),
&\beta_i^{\textsc{hs}} &=  -\tr\bigl(G_{i+1}^\tp (G_{i+1} - G_i)\bigr)/\tr\bigl(P_i^\tp (G_{i+1} - G_i)\bigr),\\
\beta_i^{\textsc{fr}} &= \tr(G_{i+1}^\tp G_{i+1})/\tr(G_i^\tp G_i),
&\beta_i^{\textsc{dy}} &= -\tr(G_{i+1}^\tp G_{i+1})/\tr\bigl(P_i^\tp (G_{i+1} - G_i)\bigr).
\end{aligned}
\end{equation}
It may appear from these formulas that we are subtracting tangent vectors from tangent spaces at different points but this is an illusion. The effective gradients $G_i$ and $G_{i+1}$ are defined by the Riemannian gradients $\grad f(Q_i) \in \T_{Q_i} \Gr(k,n)$ and $\grad f(Q_{i+1}) \in \T_{Q_{i+1}} \Gr(k,n)$ as in \eqref{eq:gradB} but they are not Riemannian gradients themselves. The formulas in \eqref{eq:cg} have in fact already accounted for the requisite parallel transports. This is another instance of the simplicity afforded by the involution model that we saw earlier in our Barzilai--Borwein step size \eqref{eq:bbM} --- our formulas in \eqref{eq:cg} are no different from the standard formulas for Euclidean space in \cite{DY,FR,HS,PR}. Contrast these with the formulas in \cite[Equations~2.80 and 2.81]{EAS}, where the parallel transport operator $\tau$ makes an explicit appearance and cannot be avoided.

\begin{algorithm}
\caption{Conjugate gradient}\label{alg:cg}
\begin{algorithmic}[1]
\State Initialize $Q_0=V_0 I_{k,n-k} V_0^\tp \in \Gr(k,n)$.
\State Compute effective gradient $G_0$ at $Q_0$ \Comment{entries $\ast$ not needed}
\[
V_0^\tp (f_{Q_0}+f_{Q_0}^\tp) V_0=  \begin{bmatrix}
* & 2G_0\\
2G_0^\tp & *
\end{bmatrix};
\]
\State initialize $P_0 = S_0 = -G_0$, $\alpha_0 =1$;

\For{$i=0,1,\dots $}

\State compute $\alpha_i$ from line search and set
\[
S_i = -\alpha_i G_i;
\]

\State update eigenbasis \Comment{effective parallel transport}
\[
V_{i+1} = V_i \exp \left( \begin{bmatrix}
0 & -S_i/2 \\
S^\tp_i/2 & 0
\end{bmatrix} \right);
\]

\State update iterate
\[
Q_{i+1} = V_{i+1}  I_{k,n-k}  V_{i+1}^\tp;
\]

\State compute effective gradient $G_{i+1}$ at $Q_{i+1}$  \Comment{entries $\ast$ not needed}
\[
V_{i+1}^\tp (f_X(Q_{i+1})+f_X(Q_{i+1})^\tp) V_{i+1}=  \begin{bmatrix}
* & 2G_{i+1} \\
2G_{i+1}^\tp & *
\end{bmatrix};
\]

\State compute Polak--Ribi\`ere conjugate step size
\[
\beta_i = \tr\bigl((G_{i+1}- G_i)^\tp G_{i+1}\bigr)/\tr(G_i^\tp G_i);
\]

\State update conjugate direction
\[
P_{i+1} = -G_{i+1}+\beta_i P_i;
\]
\EndFor
\end{algorithmic}
\end{algorithm}

Our quasi-Newton method,  given in Algorithm~\ref{alg:qn}, uses \textsc{l-bfgs} updates with two loops recursion \cite{NW}. Observe that a minor feature of Algorithms~\ref{alg:cay}, \ref{alg:sd}, \ref{alg:cg}, \ref{alg:qn} is that they do not require vectorization of matrices; everything can be computed in terms of matrix-matrix products, allowing for Strassen-style fast algorithms. While it is straightforward to replace the \textsc{l-bfgs} updates with full \textsc{bfgs}, \textsc{dfp}, \textsc{sr}{\footnotesize 1}, or Broyden class updates, doing so will require that we vectorize matrices like in Algorithm~\ref{alg:nt}.

\begin{algorithm}
\caption{Quasi-Newton with \textsc{l-bfgs} updates}\label{alg:qn}
\begin{algorithmic}[1]
\State Initialize $Q_0 =V_0 I_{k,n-k} V_0^\tp\in \Gr(k,n)$.
\For{$i=0,1,\dots $}
\State Compute effective gradient $G_i$ at $Q_i$  \Comment{entries $\ast$ not needed}
\[
V_i^\tp \bigl(f_X(Q_i)+f_X(Q_i)^\tp\bigr) V_i=  \begin{bmatrix}
* & 2G_i \\
2G_i^\tp & *
\end{bmatrix};
\]

\If {$i=0$}
	\State initialize $S_0 = -G_0$;
\Else
\State set $Y_{i-1} = G_i - G_{i-1}$ and $P=G_i$; \Comment{$P$ is temporary variable for loop}
\For{$j= i-1,\dots,\max(0, i-m) $}
\State $\alpha_j =  \tr(S_j^\tp P)/\tr(Y_j^\tp S_j)$;
\State $P = P-\alpha_j Y_j$;
\EndFor
\State set $Z = \tr(Y_{i-1}^\tp S_{i-1})/\tr(Y_{i-1}^\tp Y_{i-1})P$; \Comment{$Z$ is temporary variable for loop}
\For{$j=\max(0, i-m),\dots,i-1 $}
\State $\beta_j =  \tr(Y_j^\tp Z)/\tr(Y_j^\tp S_j)$;
\State $Z = Z+ (\alpha_j-\beta_j)S_j$;
\EndFor
\State set effective quasi-Newton step $S_i = -Z$;
\EndIf

\State update eigenbasis \Comment{effective parallel transport}
\[
V_{i+1} = V_i \exp \left( \begin{bmatrix}
0 & -S_i/2 \\
S^\tp_i/2 & 0
\end{bmatrix} \right);
\]
\State update iterate
\[
Q_{i+1} = V_{i+1} I_{k,n-k} V_{i+1}^\tp;
\]
\EndFor
\end{algorithmic}
\end{algorithm}

\subsection{Exponential-free algorithms?}

This brief section is speculative and may be safely skipped. In our algorithms, an exponential matrix $U \coloneqq \exp\bigl(\bigl[\begin{smallmatrix}
0 & B \\
-B^\tp & 0
\end{smallmatrix}\bigr]\bigr)$ is always\footnote{See steps 3, 10 in Algorithm~\ref{alg:sd}; steps 3, 7 in Algorithm~\ref{alg:nt}; steps 7, 8 in Algorithm~\ref{alg:cg};  steps 3, 20 in Algorithm~\ref{alg:qn}.} applied as a conjugation of some \emph{symmetric} matrix $X \in \mathbb{R}^{n \times n}$:
\begin{equation}\label{eq:conjugate}
X\mapsto U X U^\tp \quad \text{or} \quad X \mapsto U^\tp X U.
\end{equation}
In other words, the Givens rotations in \eqref{eq:strang} are applied in the form of \emph{Jacobi rotations} \cite[p.~477]{GVL}. For a symmetric $X$, a Jacobi rotation $X \mapsto G_{ij}(\theta) X G_{ij}(\theta)^\tp$ takes the same number (as opposed to twice the number) of floating point operations as a Givens rotation applied on the left, $X \mapsto G_{ij}(\theta) X$, or on the right, $X \mapsto X G_{ij}(\theta)$. Thus with Strang splitting the operations in \eqref{eq:conjugate} take time $12nk(n-k) $. To keep our algorithms simple, we did not take advantage of this observation.

In principle, one may avoid any actual computation of matrix exponential by simply storing the $k(n-k)$ Givens rotations in \eqref{eq:strang} without actually forming the product,  and apply them as Jacobi rotations whenever necessary. The storage of $G_{ij}(\theta)$ requires just a single floating point number $\theta$ and two indices but one would need to figure out how to update these $k(n-k)$ Givens rotations from one iteration to the next. We leave this as an open problem for interested readers.

\section{Numerical experiments}\label{sec:num}

We will describe three sets of numerical experiments, testing Algorithms~\ref{alg:cay}--\ref{alg:qn} on three different objective functions, the first two are chosen because their true solutions can be independently  determined in closed-form, allowing us to ascertain that our algorithms have converged to the global optimizer. All our codes are open source and publicly available at:
\begin{quote}
\url{https://github.com/laizehua/Simpler-Grassmannians}
\end{quote}
The goal of these numerical experiments is to compare our algorithms for the involution model  in Section~\ref{sec:algo}  with the corresponding algorithms for the Stiefel model  in \cite{EAS}.  Algorithm~\ref{alg:qn}, although implemented in our codes, is omitted from  our comparisons as quasi-Newton methods are not found in \cite{EAS}.

\subsection{Quadratic function}\label{sec:quad}

The standard test function for Grassmannian optimization is the quadratic form in \cite[Section~4.4]{EAS} which, in the Stiefel model, takes the form $\tr(Y^\tp F Y)$ for a symmetric $F \in \mathbb{R}^{n \times n}$ and $Y \in \V(k,n)$. By Proposition~\ref{prop:G3}, we write $Q = 2YY^\tp - I$, then $\tr(Y^\tp F Y) = \bigl(\tr(FQ) + \tr(F)\bigr)/2$. Therefore, in the involution model, this optimization problem takes an even simpler form
\begin{equation}\label{eq:quad}
f(Q) = \tr(FQ)
\end{equation}
for $Q \in \Gr(k,n)$. What was originally quadratic in the Stiefel model becomes linear in the involution model. The minimizer of $f$, 
\[
Q_* \coloneqq \argmin \bigl\{ \tr(FQ) : Q^\tp Q = I, \; Q^\tp = Q, \; \tr(Q) = 2k - n \bigr\},
\]
is given by $Q_* =\Pi V I_{k,n-k} V^\tp \Pi^\tp$ where
\[
\Pi = \begin{bmatrix}  & & 1 \\ & \Ddots & \\ 1 & & \end{bmatrix}\qquad\text{and}\qquad \frac{F + F^\tp}{2} = V D V^\tp
\]
is an eigendecomposition with eigenbasis $V \in \O(n)$ and eigenvalues $D \coloneqq \diag (\lambda_1,\dots,\lambda_n)$ in descending order.
This follows from essentially the same argument\footnote{Recall also that for any real numbers $a_1 \le \dots \le a_n$, $b_1 \le \dots \le b_n$, and  any permutation $\pi$, one always have that $a_1 b_n + a_2 b_{n-1} + \dots + a_n b_1 \le a_1 b_{\pi(1)} +a_2 b_{\pi(2)} + \dots + a_n b_{\pi(n)} \le a_1 b_1 + a_2 b_2 + \dots + a_n b_n$.} used in the proof of Lemma~\ref{lem:approx} and
the corresponding minimum is $f(Q_*) = -\lambda_1 - \dots - \lambda_k + \lambda_{k+1} + \dots + \lambda_n$.

For the function $f(Q) = \tr(FQ)$, the effective gradient $G_i \in \mathbb{R}^{k \times (n-k)}$ in Algorithms~\ref{alg:sd}, \ref{alg:cg}, \ref{alg:qn} at the point $Q_i = V_i I_{k,n-k} V_i^\tp \in \Gr(k,n)$ is given by
\[
V_i^\tp F V_i=  \begin{bmatrix}
A & G_i \\
G_i^\tp & C
\end{bmatrix}.
\]
The matrices $A \in \mathbb{R}^{k \times k}$ and $C \in \mathbb{R}^{(n-k) \times (n-k)}$ are not needed for Algorithms~\ref{alg:sd}, \ref{alg:cg}, \ref{alg:qn} but they are required in Algorithm~\ref{alg:nt}. Indeed, the effective Newton step $S_i  \in \mathbb{R}^{k \times (n-k)}$  in Algorithm~\ref{alg:nt} is obtained by solving the Sylvester equation
\[
AS_i - S_iC = 2G_i.
\]
To see this, note that by Proposition~\ref{prop:Hess},  for any $B  \in \mathbb{R}^{k \times (n-k)}$,
\begin{align*}
\Hess f(Q_i) \bigg(V_i\begin{bmatrix}
0 & B \\
B^\tp & 0
\end{bmatrix} V_i^\tp, V_i \begin{bmatrix}
0 & S_i \\
S_i^\tp & 0
\end{bmatrix} V_i^\tp \bigg) 
&= -\frac{1}{2}\tr\bigg( \begin{bmatrix}
A & G_i \\
G_i^\tp & C
\end{bmatrix} \begin{bmatrix}
XS_i^\tp + S_i B^\tp & 0 \\
0 & -B^\tp S_i - S_i^\tp B
\end{bmatrix}\bigg) \\
& = -\tr\bigl(B^\tp(AS_i - S_iC)\bigr),
\end{align*}
and to obtain the effective Newton step  \eqref{eq:ns}, we simply set the last term to be equal to  $-2\tr(B^\tp G_i)$.

Figure~\ref{fig:plot1} compares the convergence behaviors of the algorithms in \cite{EAS} for the Stiefel model and our Algorithms~\ref{alg:sd}, \ref{alg:nt}, \ref{alg:cg} in the involution model: steepest descent with line search (\textsc{gd})  and with Barzilai--Borwein step size (\textsc{bb}),  conjugate gradient (\textsc{cg}), and Newton's method (\textsc{nt}) for $k=6$, $n =16$. We denote the $i$th iterate in the Stiefel and involution models by $Y_i$ and $Q_i$ respectively --- note that $Y_i$ is a $16 \times 6$ matrix with orthonormal columns whereas $Q_i$ is a $16 \times 16$ symmetric orthogonal matrix. All algorithms are fed the same initial point obtained from 20 iterations of Algorithm~\ref{alg:cay}. Since we have the true global minimizer in closed form, denoted by $Y_*$ and $Q_*$ in the respective model, the error is given by geodesic distance to the true solution. For convenience we compute  $\| Y_i Y_i^\tp - Y_* Y_*^\tp \|_\F$ and $\| Q_i - Q_* \|_\F$, which are constant multiples of the chordal distance \cite[Table~2]{YL} (also called projection $\mathsf{F}$-norm \cite[p.~337]{EAS}) and are equivalent, in the sense of metrics, to the geodesic distance. Since we use a log scale, the vertical axes of the two graphs in Figure~\ref{fig:plot1} are effectively both geodesic distance and, in particular, their values may be compared. The conclusion is clear: While Algorithms~\ref{alg:sd} (\textsc{bb}) and \ref{alg:nt} (\textsc{nt}) in the involution model attain a level of accuracy on the order of machine precision, the corresponding algorithms in the Stiefel model do not. The reason is numerical stability, as we will see next.
\begin{figure}[h]
\caption{Convergence behavior of algorithms in the Stiefel and involution models.}
\label{fig:plot1}
\centering
\includegraphics[width=0.75\textwidth]{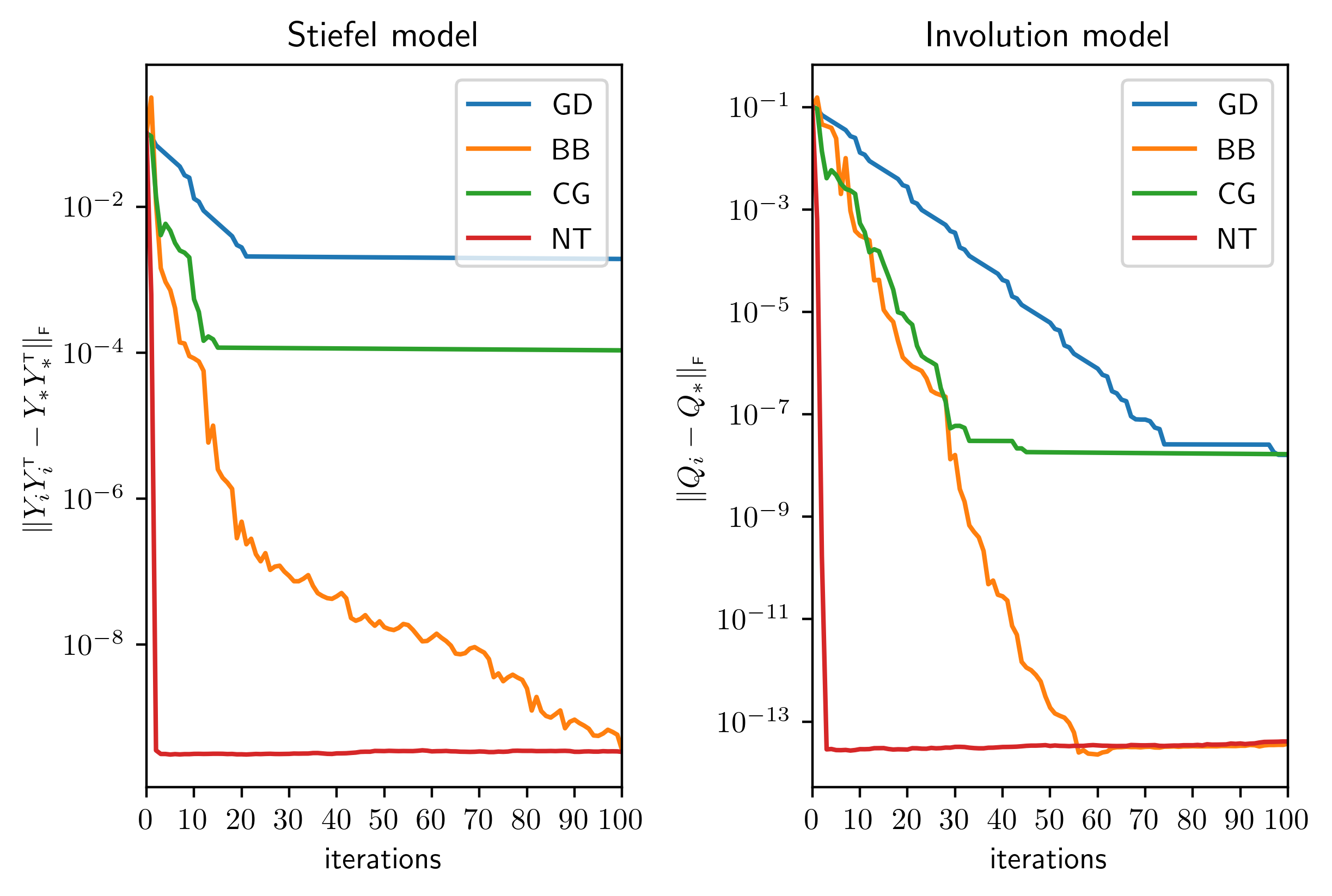}
\end{figure}

Figure~\ref{fig:plot2} shows the loss of orthogonality for various algorithms in the Stiefel and involution models, measured respectively by $\|Y_i^\tp Y_i - I\|_\F$ and $\|Q_i^2 - I\|_\F$. In the Stiefel model, the deviation from orthogonality $\|Y_i^\tp Y_i - I\|_\F$ grows exponentially. In the worst case, the \textsc{gd} iterates $Y_i$, which of course ought to be of rank $k=6$, actually converged to a rank-one matrix. In the involution model, the deviation from orthogonality  $\|Q_i^2 - I\|_\F$ remains below $10^{-13}$ for all algorithms --- the loss-of-orthogonality is barely noticeable.
\begin{figure}[h]
\caption{Loss of orthogonality in Stiefel and involution models.}
\label{fig:plot2}
\centering
\includegraphics[width=0.75\textwidth]{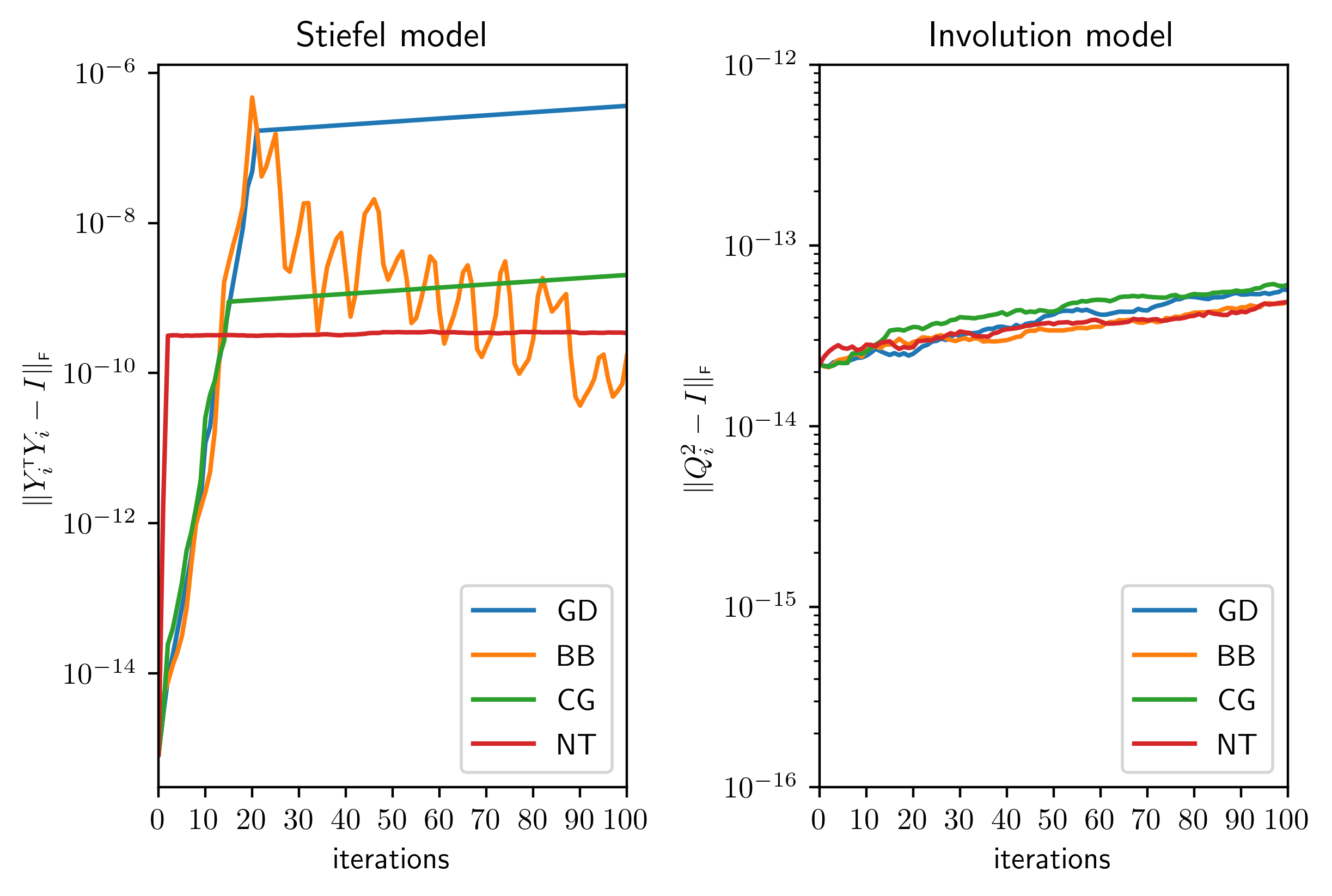}
\end{figure}

A closer inspection of the algorithms for \textsc{nt} \cite[p.~325]{EAS} and \textsc{cg} \cite[p.~327]{EAS} in the Stiefel model reveals why: A point $Y_i$ and the gradient $G_i$  at that point are highly dependent on each other --- an $\varepsilon$-deviation from orthogonality in $Y_i$ results in an $\varepsilon$-error in $G_i$ that in turn becomes a $2\varepsilon$-deviation from orthogonality in $Y_{i+1}$, i.e., one loses orthogonality at an exponential rate. We may of course reorthogonalize $Y_i$ at every iteration in the Stiefel model to artificially enforce the orthonormality of its columns but this incurs additional cost and turns a Riemannian algorithm into a retraction algorithm, as reorthogonalization of $Y_i$ is effectively a \textsc{qr} retraction.

Contrast this with the involution model: In Algorithms~\ref{alg:nt} (\textsc{nt}) and \ref{alg:cg} (\textsc{cg}), the point $Q_i$ and the effective gradient $G_i$ are both computed directly from the  eigenbasis $V_i$, which is updated to $V_{i+1}$ by an orthogonal matrix, or a sequence of Givens rotations if one uses Strang splitting as in  \eqref{eq:strang}. This introduces a small (constant order) deviation from orthogonality each step. Consequently, the deviation from orthogonality  at worst grows linearly.

\subsection{Grassmann Procrustes problem}\label{sec:procrustes}

Let $k,m,n \in \mathbb{N}$ with $k \le n$. Let $A \in \mathbb{R}^{m \times n}$ and $B \in \mathbb{R}^{m \times k}$. The minimization problem
\begin{equation}\label{eq:GP}
\min_{Q^\tp Q = I} \lVert A - BQ \rVert_\F,
\end{equation}
is called the Stiefel Procrustes problem \cite[Section~3.5.2]{EAS} and the special case $k=n$ is the usual orthogonal Procrustes problem \cite[Section~6.4.1]{GVL}. Respectively, these are
\[
\min_{Q \in \V(k,n)} \lVert A - BQ \rVert_\F \qquad \text{and}\qquad \min_{Q \in \O(n)} \lVert A - BQ \rVert_\F.
\]
One might perhaps wonder if there is also a \emph{Grassmann Procrustes problem}
\begin{equation}\label{eq:GP}
\min_{Q \in \Gr(k,n)} \lVert A - BQ \rVert_\F. 
\end{equation}
Note that here we require $m = n$.
In fact, with the involution model for $\Gr(k,n)$, the problem \eqref{eq:GP} makes perfect sense. The same argument in the proof of Lemma~\ref{lem:approx} shows that the minimizer $Q_*$ of \eqref{eq:GP} is given by $Q_* = V I_{k,n-k} V^\tp$ where
\[
\frac{A^\tp B + B^\tp A}{2} = V D V^\tp
\]
is an eigendecomposition with eigenbasis $V \in \O(n)$ and eigenvalues $D \coloneqq \diag (\lambda_1,\dots,\lambda_n)$ in descending order. The convergence and loss-of-orthogonality behaviors for this problem are very similar to those in Section~\ref{sec:quad} and provides further confirmation for the earlier numerical results. The plots from solving \eqref{eq:GP} for arbitrary $A,B$ using any of Algorithms~\ref{alg:sd}--\ref{alg:qn} are generated in our codes but given that they are nearly identical to Figures~\ref{fig:plot1} and \ref{fig:plot2} we omit them here.

\subsection{Fr\'echet mean and Karcher mean}

Let $Q_1,\dots,Q_m \in \Gr(k,n)$ and consider the sum-of-square-distances minimization problem:
\begin{equation}\label{eq:km}
\min_{Q \in \Gr(k,n)} \sum_{j=1}^{m} d^2(Q_j,Q),
\end{equation}
where $d$ is the geodesic distance in \eqref{eq:geodist}. The global minimizer of this problem is called the \emph{Fr\'echet mean} and a local minimizer is called a \emph{Karcher mean} \cite{Kar2}. For the case $m =2$, a Fr\'echet mean is the midpoint, i.e., $t =1/2$, of the geodesic connecting $Q_1$ and $Q_2$ given by the closed-form expression in Proposition~\ref{prop:geo}.
The objective function $f$ in \eqref{eq:km} is differentiable almost everywhere\footnote{$f$ is nondifferentiable only when $Q$ falls on the cut locus of $Q_i$ for some $i$ but the union of all cut loci of $Q_1,\dots,Q_m$ has codimension $\ge 1$.}   with its Riemannian gradient \cite{Kar} given by
\[
\grad f(Q) = 2 \sum_{j=1}^{m}  \log_Q (Q_j),
\]
where the logarithmic map is as in Corollary~\ref{cor:log}. To the best of our knowledge, there is no simple expression for $\Hess f(Q)$ and as such we exclude Newton method from consideration below.

\begin{figure}[h]
\caption{Convergence behavior of algorithms in the Stiefel and involution models.}
\label{fig:plot3}
\centering
\includegraphics[width=0.75\textwidth]{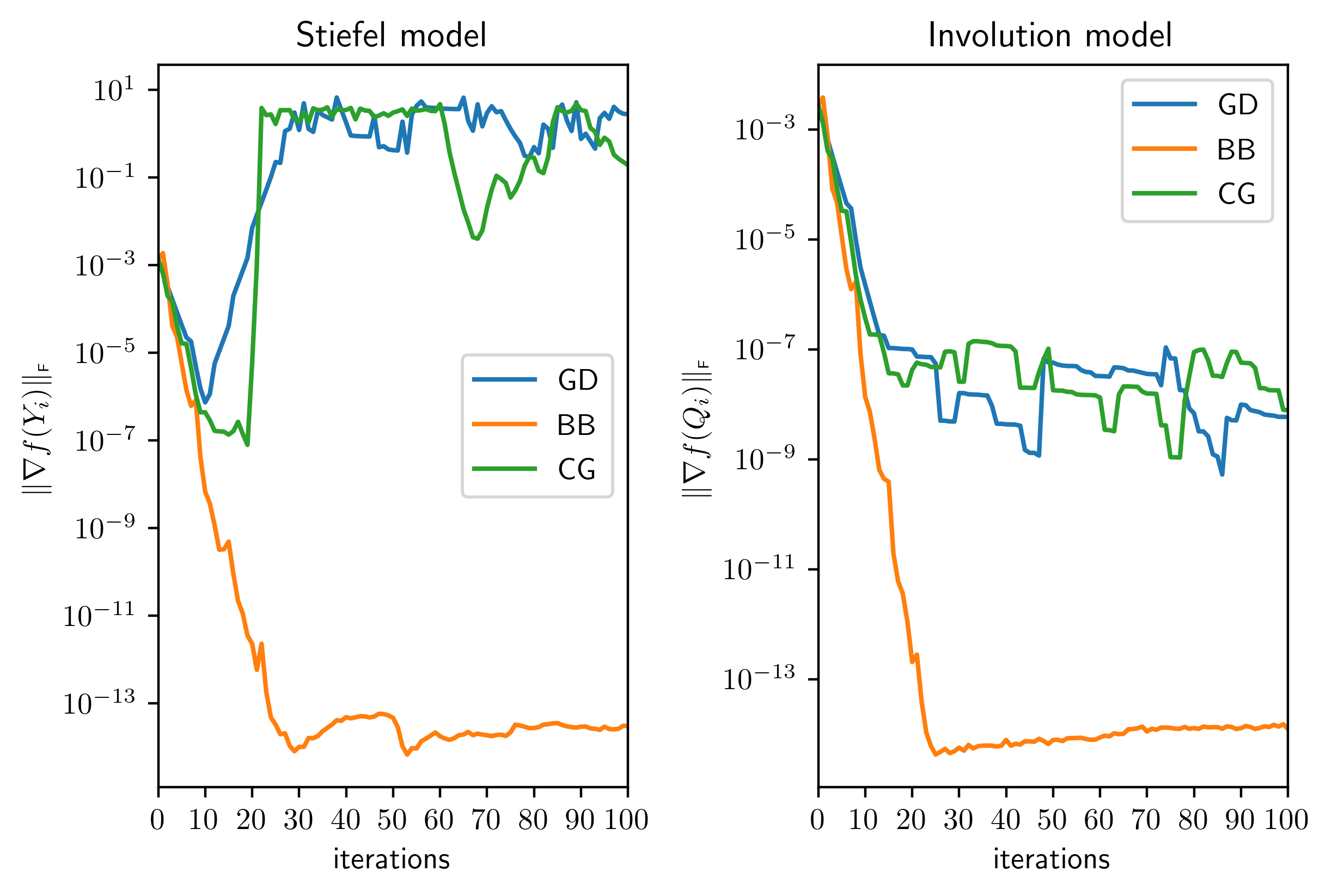}
\end{figure}

We will set $k = 6$, $n = 16$, and $m =3$. Unlike the problems in Sections~\ref{sec:quad} and \ref{sec:procrustes}, the problem in \eqref{eq:km} does not have a closed-form solution when $m > 2$. Consequently we quantify convergence behavior in Figure~\ref{fig:plot3} by the rate gradient goes to zero. The deviation from orthogonality is quantified as in Section~\ref{sec:quad} and shown in 
Figure~\ref{fig:plot4}.  The instability of the algorithms in the Stiefel model is considerably more pronounced here --- both \textsc{gd} and \textsc{cg} failed to converge to a stationary point as we see in Figure~\ref{fig:plot3}.  The cause, as revealed by Figure~\ref{fig:plot4}, is a severe loss-of-orthogonality that we will elaborate below.

The expression for geodesic distance $d(Y,Y')$ between two points $Y,Y'$ in the Stiefel model (see \cite[Section~3.8]{AMS} or \cite[Equation~7]{YL}) is predicated on the crucial assumption that each of these matrices has orthonormal columns. As a result, a moderate deviation from orthonormality in an iterate $Y$ leads to vastly inaccurate values in the objective function value $f(Y)$, which is a sum of $m$ geodesic distances \emph{squared}. This is reflected in the graphs  on the left of Figure~\ref{fig:plot3} for the \textsc{gd} and \textsc{cg} algorithms, whose step sizes  come from line search and depend on these function values. Using the \textsc{bb} step size, which does not depend on objective function values, avoids the issue. But for \textsc{gd} and \textsc{cg}, the reliance on inaccurate function values leads to further loss-of-orthogonality, and when the columns of an iterate $Y$ are far from orthonormal, plugging $Y$ into the expression for gradient simply yields a nonsensical result, at times even giving an \emph{ascent} direction in a minimization problem.\footnote{This last observation is from a plot of the function values that comes with our code but is not included here.}

For all three algorithms in the involution model, the deviation from orthogonality in the iterates is kept at a negligible level of under $10^{-13}$ over the course of 100 iterations.

\begin{figure}[h]
\caption{Loss of orthogonality in the Stiefel and involution models.}
\label{fig:plot4}
\centering
\includegraphics[width=0.75\textwidth]{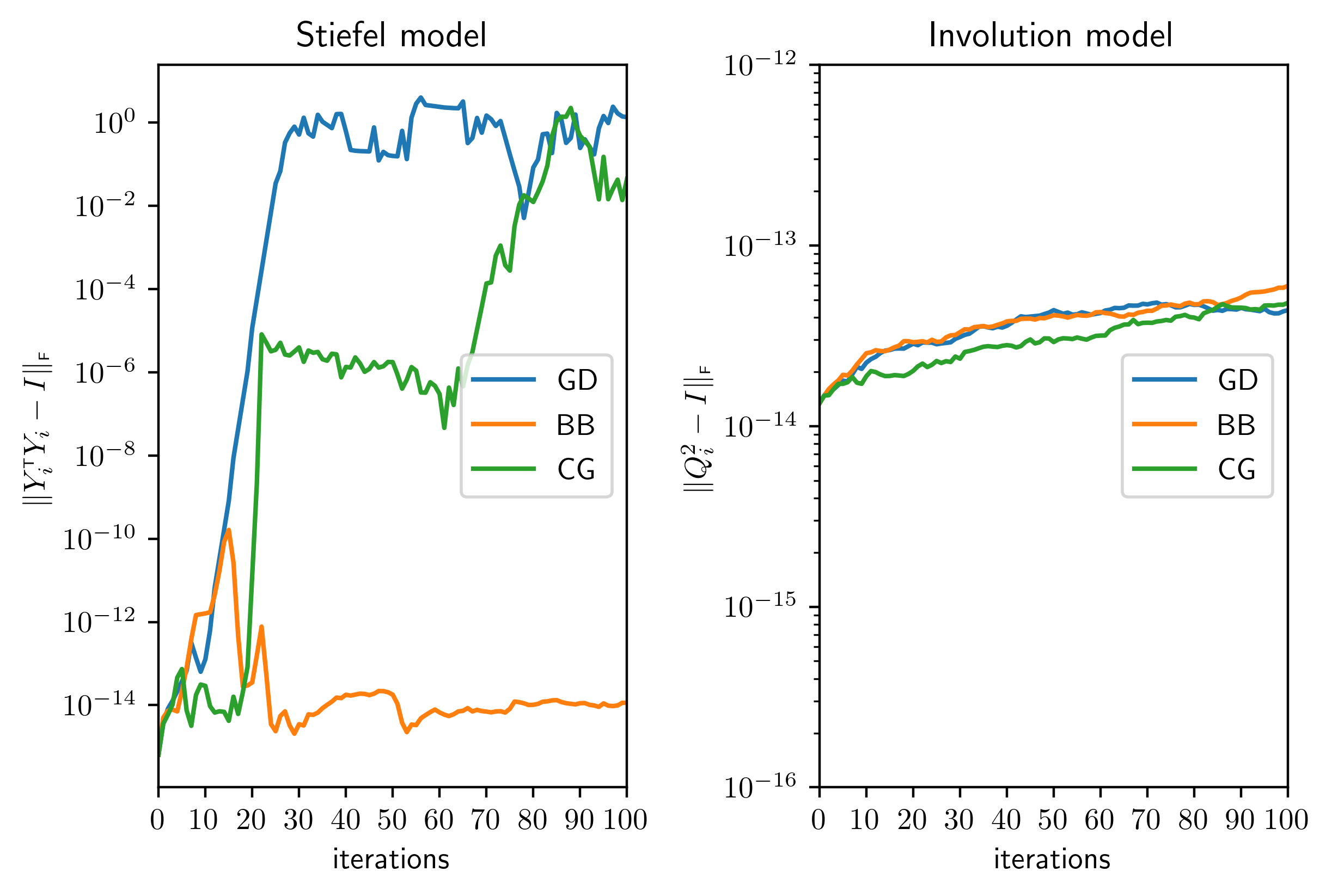}
\end{figure}

\subsection*{Acknowledgment}

We would like to acknowledge the intellectual debt we owe to \cite{AMS,EAS,HHT}. The work in this article  would not have been possible without drawing from their prior investigations.

ZL is supported by a Neubauer Family Distinguished Doctoral Fellowship from the University of Chicago. LHL is supported by  NSF IIS 1546413, DMS 1854831, and the Eckhardt Faculty Fund. KY is supported by NSFC Grant no.~11688101, NSFC Grant no.~11801548 and National Key R\&D Program of China Grant no.~2018YFA0306702.

\end{document}